\documentclass[microtype]{gtpart}     
\agtart
\usepackage[all]{xy}
\usepackage{amscd}
\usepackage{graphicx}
\usepackage[leqno]{amsmath}
\usepackage{amssymb}
\usepackage{amsthm}
\usepackage[all]{xy}
\usepackage{pinlabel}

\usepackage{epsf,times}
\usepackage{amsthm}
\usepackage{amsbsy}
\usepackage{bm}

\usepackage[usenames,dvipsnames]{color}
\usepackage[normalem]{ulem}

\newtheorem{thm}{Theorem}[section]
\newtheorem{theorem}[thm]{Theorem} 
\newtheorem{notation}[thm]{Notation}
\newtheorem{lemma}[thm]{Lemma}
\newtheorem{cor}[thm]{Corollary}
\newtheorem{prop}[thm]{Proposition}
\newtheorem{con}[thm]{Conjecture}
\newtheorem{remark}[thm]{Remark}

\theoremstyle{definition}
\newtheorem{definition}[thm]{Definition}

\newtheorem{question}[thm]{Question}

\theoremstyle{plain}

\newcommand{\Int}{\mathrm{int}}

\title{Persistently foliar composite knots}

\author{Charles Delman}
\givenname{Charles}
\surname{Delman}
\address{Department of Mathematics, Eastern Illinois University, Charleston, IL 61920}
\email{cidelman@eiu.edu}

\author{Rachel Roberts}
\givenname{Rachel}
\surname{Roberts}
\address{Department of Mathematics, Washington University, St.  Louis, MO 63130}
\email{roberts@wustl.edu}
\keyword{taut foliation}
\keyword{persistently foliar knot}
\keyword{L-space}
\keyword{L-space knot}
\keyword{composite knot}
\keyword{spine}
\keyword{branched surface}

\subject{primary}{msc2010}{57M50}

\arxivreference{1905.04838 }

\begin{document}



\begin{abstract}    
A knot $\kappa$ in $S^3$ is \emph{persistently foliar} if, for each non-trivial boundary slope, there is a  co-oriented taut foliation meeting the boundary of the knot complement transversely in a foliation by curves of that slope.  For rational slopes, these foliations may be capped off by disks to obtain a  co-oriented taut foliation in every manifold obtained by non-trivial Dehn surgery on that knot.   We show that any composite knot with a persistently foliar summand is persistently foliar and that any nontrivial connected sum of fibered knots is persistently foliar.   As an application, it follows that any composite knot in which each of two summands is fibered or at least one summand is nontorus alternating or Montesinos is persistently foliar.   

We note that, in constructing foliations in the complements of fibered summands, we build branched surfaces whose complementary regions agree with those of Gabai's product disk decompositions, except for the one containing the boundary of the knot complement.  It is this boundary region which provides for persistence.

\end{abstract}

\maketitle

\section{Introduction}

Co-oriented taut foliations  play an important role in the study of 3-manifolds.   Recently, the search for co-oriented taut foliations in 3-manifolds has been informed by the L-space conjecture    \cite{OzSz3,boyergordonwatson,juhasz3}, which states that an irreducible space that is not an L-space necessarily contains a co-oriented taut foliation.  Considering manifolds obtained by Dehn surgery on $S^3$, a knot $\kappa$ is called an L-space knot if some non-trivial Dehn surgery on $\kappa$ yields an L-space.   A knot $\kappa$ is \emph{persistently foliar} if, for each boundary slope, there is a co-oriented taut  foliation meeting the boundary of the knot complement transversely in a foliation by curves of that slope. For rational slopes, these foliations may be capped off by disks to obtain a co-oriented taut  foliation in every manifold obtained by Dehn surgery on that knot.  In this context, we  propose  the L-space knot conjecture:  
 \begin{con}[L-space knot conjecture]
A knot is persistently foliar if and only if it is not an L-space knot and has no reducible surgeries. 
\end{con}
Krcatovich \cite{Krc} has proven that nontrivial connected sums of knots are never L-space knots.  In this paper, we prove that many  composite knots  are also persistently foliar, as detailed in the results described below.   It follows that any such knot $\kappa$ satisfies the L-space knot conjecture, and any 3-manifold obtained by Dehn surgery along $\kappa$ satisfies the L-space conjecture.

Let $\kappa$ be any knot in $S^3$, and fix a regular neighbourhood $N(\kappa)$ of   $\kappa$. Set $X_{\kappa}=S^3\setminus\Int N(\kappa)$. 
Parametrize $\partial N(\kappa)$ as $S^1\times S^1$ so that  $\{1\} \times S^1$ represents the meridian, and $S^1\times \{1\}$ represents the longitude, of $\kappa$. A lamination of $\partial N(\kappa)$ has slope $m\in \mathbb{R}P^1$  if it is isotopic to the image of lines of slope $m$ under the universal covering map $\mathbb R^2\to S^1\times S^1: (s,t) \mapsto (e^{2\pi i s},e^{2\pi i t} )$. The slope $1/0$ is called the \emph{trivial slope}.

More generally, given any oriented 3-manifold $M$ with a single torus boundary component, which we give the standard orientation induced by the orientation of $M$, define the set of \emph{slopes} on $\partial M$ to be the set of isotopy classes of unoriented (simple) curves on $\partial M$.  In the case that $M$ is fibered over $S^1$ with fiber $F$, we distinguish $\partial F$ as the longitude of $\partial M$ and denote it by $\lambda$.  In this context we define a \emph{meridian} to be any curve having a single point of minimal transverse intersection with $\lambda$.  Once a distinguished meridian is chosen (see Section \ref{section: prelim}), each slope may be identified with a point in $\mathbb{R}P^1$.  A different choice of meridian results in a parabolic shift, fixing $0$ (the longitudinal slope), of the associated points of $\mathbb{R}P^1$;  since a parabolic shift preserves the cyclic ordering, we may speak of an interval of slopes (with given endpoints and containing a given third slope in its interior) independently of this choice.

\begin{definition}
A foliation $\mathcal{F}$ \emph{strongly realizes} a slope if $\mathcal{F}$ intersects $\partial N(k)$ transversely in a foliation by curves of that slope.
\end{definition}

\begin{remark}
Note that no co-oriented taut foliation strongly realizes the meridian of a knot in $S^3$, since $S^3$ is simply connected.
\end{remark}

We proceed as follows.  First we show that connected sums behave well with respect to strong realization of slopes:

\noindent \textbf{Proposition~\ref{anyproblem}}.  \emph{Suppose $\kappa=\kappa_1\#\kappa_2$  is a connected sum of knots in $S^3$. If the slope $m$ along $\kappa_1$ is strongly realized, then so is the slope $m$ along $\kappa$.} 


\noindent \textbf{Corollary~\ref{consum}}  \emph{Suppose $\kappa=\kappa_1\#\cdots \#\kappa_n$ is a connected sum of  knots. If at least one of the $\kappa_i$ is persistently foliar, then so is $\kappa$.}

We next show that connected sums of fibered knots are persistently foliar and therefore satisfy the L-space Knot Conjecture:

\noindent \textbf{Theorem~\ref{main}} \emph{ Suppose $\kappa_1$ and $\kappa_2$ are nontrivial fibered knots in $S^3$.
Any nontrivial slope on $\kappa=\kappa_1\# \kappa_2$ is strongly realized by a co-oriented taut foliation  that  has a single   minimal set, disjoint from   $\partial N(\kappa)$.  Hence $\kappa_1\#\kappa_2$  is persistently foliar.
}

Combining the results above with those of \cite{DR2, DR3, DR4}, we obtain:

\noindent \textbf{Corollary~\ref{onesummandenough}}  \emph{Suppose $\kappa=\kappa_1\#\cdots \#\kappa_n$ is a connected sum of  knots. If at least one of the $\kappa_i$ is a nontorus alternating or Montesinos knot or a connected sum of fibered knots, then $\kappa$ is persistently foliar.}

\noindent \textbf{Corollary~\ref{alternatingmontesinos2}}  \emph{Suppose $\kappa$ is a composite knot with a summand that is a nontorus alternating or Montesinos knot or the connected sum of two fibered knots,  and $\widehat{X}_{\kappa}$ is a manifold obtained by non-trivial Dehn surgery along $\kappa$.   Then $\widehat{X}_{\kappa}$ contains a co-oriented taut foliation;  hence, $\kappa$ satisfies the L-space Knot Conjecture.}
 
Since connected sums of fibered knots are necessarily fibered (\cite{Stallings};  for a geometric argument, see \cite{GabMurasugi}),
we can contrast the co-oriented taut foliations constructed in this paper  with  those constructed  in \cite{Rfib1,Rfib2}. First, we combine  some results found in \cite{Rfib2} and restate them using the language of Honda, Kazez and Mati\'c   \cite{HKM1}:

\begin{thm} \cite{Rfib2} \label{oldresult} Suppose $\kappa$ is any nontrivial fibered knot in $S^3$, with monodromy $\phi$.   Exactly  one of the following is true:
\begin{enumerate}
\item  $\phi$ is right-veering,   and  for some $1\le r< \infty$, any slope in $(-\infty,r)$ is strongly realized by a minimal co-oriented taut  foliation. 
\item  $\phi$ is left-veering,  and  for some $1\le r< \infty$, any slope in $(-r, \infty)$ is strongly realized by a  minimal co-oriented taut  foliation.  
\item  Any nontrivial slope is strongly realized by a minimal co-oriented taut foliation.
\end{enumerate}
\end{thm}

A nontrivial connected sum  of fibered knots   has right-veering (left-veering, respectively) monodromy only if each nontrivial component has right-veering (left-veering, respectively) monodromy. Hence, if $\kappa_1$ and $\kappa_2$ are nontrivial fibered knots in $S^3$  with monodromies that are neither both left-veering nor both right-veering, then the construction of \cite{Rfib2} yields co-oriented taut foliations that strongly realize all nontrivial boundary slopes. In Section~\ref{Remaining}, we focus on this case, and { prove that the methods of this paper yield new constructions of } co-oriented taut foliations.     Recall that a  minimal set is called \textit{genuine} if there is at least one region complementary to the minimal set that is  is not a product \cite{GabaiKazez}.

\textbf{Theorem~\ref{richcase}} \textit{Suppose that $X$ is a fibered 3-manifold, with  fiber $F$ a compact oriented surface with connected boundary, and orientation-preserving monodromy $\phi$.  If there is a tight arc $\alpha$ so that the corresponding  product disk $D(\alpha)$ has transition arcs of opposite sign, then there  is a  co-oriented taut foliation $\mathcal F_r$  that strongly realizes slope $r$ for all slopes except   $\mu$, the distinguished meridian. Furthermore, each $\mathcal F_r$ extends to a co-oriented taut  foliation $\widehat{\mathcal F}_r$ in $\widehat{X}(r)$, the closed 3-manifold obtained by Dehn filling along $r$, and  when  $r$ intersects the meridian efficiently in at least two points, the minimal set of $\widehat{\mathcal F}(r)$  is genuine.}

In contrast, the foliations constructed in \cite{Rfib1,Rfib2} are minimal up to Denjoy blowdown. Hence, when  the  foliations  $\widehat{\mathcal F}_r$   have  genuine  minimal set,  they cannot be isotopic to the  foliations constructed in \cite{Rfib1,Rfib2}.  However, it is possible that these foliations are equivalent under some coarser notion of equivalence.

\begin{question} \label{Q:same} Suppose $\kappa=\kappa_1\#\kappa_2$ is a nontrivial connected sum   that is fibered, and let $M$   be obtained by nontrivial  Dehn surgery along $\kappa$. Let $\mathcal F$, and $\mathcal F'$ be co-oriented taut foliations in $M$, with $\mathcal F$ constructed as in \cite{Rfib2} and $\mathcal F'$ constructed as described in this paper.
\begin{enumerate}
\item Are $\mathcal F$ and $\mathcal F'$ coarsely isotopic    \cite{GabKneser}?
\item Are $\mathcal F$ and $\mathcal F'$ transverse to a common smooth flow?
\item Are $\mathcal F$ and $\mathcal F'$ transverse to nowhere vanishing vector fields $\mathbf{v}$ and $\mathbf{v}'$, respectively, that represent a common $\mbox{Spin}^{\mbox{c}}$-structure $\mathfrak{s}\in \mbox{Spin}^{\mbox{c}}(M)$   \cite{OzSz2}?
\item If yes to (2), do weakly symplectically fillable contact structures  $\xi$, $\xi'$  approximating, respectively, $\mathcal F$ and $\mathcal F'$  have common contact invariant $c(\xi)=c(\xi')\in \widehat{\mbox{HF}}(M,\mathfrak{s})$   \cite{OzSz5,HKM}? 
\end{enumerate}
\end{question}
The construction of co-oriented taut foliations in this paper, as well as in \cite{DR2,DR3, DR4}, involves making choices of  spine and co-orientation on the branches of  a spine  chosen.   
It seems likely that different choices can lead to co-orientable taut foliations that are not isotopic (even up to reversing the co-orientation), and hence (1)--(4) of Question~\ref{Q:same} apply.   

Note that work of Ghiggini  \cite{Ghiggini} and Ni \cite{Ni1,Ni2} (see also \cite{juhasz1, juhasz2})  establishes that  an L-space knot is necessarily fibered.  Hence, conjecturally, any non-fibered knot in $S^3$ is persistently foliar.  Restricting attention to fibered knots permits us to minimize use of the theory of sutured manifolds and thus to emphasize the simplicity of the construction.  In a future paper \cite{DR5}, we discuss more general conditions that allow for the construction of co-oriented taut foliations that strongly realize all boundary slopes except one.  In particular,  we make the following conjecture:

\begin{con}
Every composite knot is persistently foliar.
\end{con}

 All constructions of co-oriented taut foliations found in this paper are adaptions  of the  \emph{pure arrow} type construction found in \cite{DR2}.  Among the constructions of persistent families of co-oriented taut foliations found in \cite{DR2, DR3,DR4,DR5, DR6}, these are the ones closest to the sutured manifold constructions introduced by Gabai.

  \section{Acknowledgements}
 
 The work for this paper began at Casa Matem\'atica Oaxaca (CMO), where the second author attended the workshop  \textit{Thirty Years of Floer Theory for 3-Manifolds}.
 We thank Michel Boileau for calling our attention to the specialness of connected sums, and Casa Matem\'atica Oaxaca (CMO) for their hospitality.
 
This work was partially supported by National Science Foundation Grant DMS-1612475.
 
 \section{Preliminary Definitions} \label{section: prelim}

\subsection{Fibered knots  and product disks}  

A knot $\kappa$ in $S^3$ is \emph{fibered}   if   the knot complement  $X_{\kappa}$ is homeomorphic to $$F\times [0,1]/\sim,$$ where $(x,1)\sim (\phi(x),0)$ for some compact orientable surface $F$ and homeomorphism $\phi:F\to F$. In this case, $F$ is called a \emph{fiber} of $\kappa$, and $\phi$ the \emph{monodromy map} of the fibering.  The homeomorphism type of $F\times [0,1]/\sim$ is dependent only on the isotopy class of $\phi$.  

We will always assume that a fibered knot $\kappa$ and its fiber $F$  are consistently oriented; namely, $F$  is oriented and $\kappa$ is isotopic to $\lambda = \partial F$ as an oriented manifold.    We also assume an orientation of $S^3$, with the induced normal orientation on $F$ equal to the increasing orientation on the $[0,1]$ factor of $F \times [0,1]$. 

We next  state some results in the context of fibered 3-manifolds, rather than restricting attention to knot complements. 
Let $X$ denote the fibered 3-manifold  $$F\times [0,1]/\sim,$$ where $(x,1)\sim (\phi(x),0)$, for some compact orientable surface $F$ and orientation preserving homeomorphism $\phi:F\to F$.  We restrict attention to the case that $F$ has a single boundary component.
As described in \cite{Rfib2} and \cite{KR}, there is a canonical choice of meridian $\mu$, and hence a canonical parametrization of $\partial X$ as $S^1\times S^1$ so that $S^1\times \{1\}$ is isotopic to $\lambda$, and $\{1\} \times S^1$ is isotopic to $\mu$. When $X=X_{\kappa}$ is a knot complement in $S^3$, this canonical choice agrees with the standard one, although this is not necessarily the case for a knot complement in a general 3-manifold. For completeness, we give below a purely topological description of this distinguished meridian $\mu$.  

Two properly embedded arcs intersect \emph{efficiently} (\emph{efficiently rel endpoints}) if any intersections are transverse and no isotopy through properly embedded arcs (rel endpoints) reduces the number of points of intersection. A pair of properly embedded arcs $(\alpha,\beta)$ is \emph{tight} if either $\alpha=\beta$ (as unoriented arcs) or if $\alpha$ and $\beta$ are non-isotopic and intersect efficiently. Given a properly embedded arc $\alpha$, we may isotope $\phi$ so that $(\alpha,\phi(\alpha))$ is a tight pair; in this case, we say that $\alpha$ is \emph{tight (with respect to $\phi$)}. Given a tight pair $(\alpha,\beta)$, it is clear that $(\phi(\alpha), \phi(\beta))$ is also a tight pair;  furthermore, we may isotope $\phi$ so that the arcs $\alpha,\beta,\phi(\alpha)$, and $\phi(\beta)$ are pairwise tight.  Indeed,   any finite collection of properly embedded oriented arcs can be isotoped to be pairwise tight, and $\phi$ then isotoped so that the collection of arcs together with their images under $\phi$ are pairwise tight. When working with a finite collection of arcs, we will henceforth assume that these arcs and then $\phi$ have been isotoped in this way.

\begin{notation} \label{defntransitionarc1}  
Given any    properly embedded arc $\alpha \in F$, with endpoints $\alpha(0)$ and $\alpha(1)$ in $\partial F$, let   $D(\alpha)$  be the image of $\alpha \times [0,1]$, and let $\delta_i(\alpha)$  be the image of  $\alpha(i) \times [0,1]$, for $i = 0,1$, under the quotient map $F \times [0,1] \rightarrow F \times [0,1]/\sim$.  Identify $\alpha$ with the image of $\alpha \times \{0\}$ and $\phi(\alpha)$ with the image of $\alpha \times \{1\}$;  thus,   $\partial D(\alpha) = \alpha \cup \phi(\alpha) \cup \delta_0(\alpha) \cup \delta_1(\alpha)$.    
\end{notation}

Now consider an oriented properly embedded tight essential arc $\alpha$. If $\alpha=\phi(\alpha)$ (as oriented arcs), set $\mu=\delta_0(\alpha)$. Suppose that $\alpha\ne\phi(\alpha)$ (as oriented arcs). The endpoints of $\delta_0(\alpha)$ cut $\lambda$ into two open arcs, $\rho_1(\alpha)$ and $\rho_2(\alpha)$ say. The simple closed curves $\mu_1(\alpha)=\delta_0(\alpha)\cup\rho_1(\alpha)$ and $\mu_2(\alpha)=\delta_0(\alpha)\cup\rho_2(\alpha)$ are meridians satisfying $|\langle \mu_1(\alpha),\mu_2(\alpha)\rangle|=1$.

We observe the following, which will be useful in the definition of $\mu$ given below. If $\partial M$ is parametrized as $S^1\times S^1$ so that  $\{1\}\times S^1$ represents the meridian $\mu_i(\alpha)$ and $S^1\times \{1\}$ represents $\lambda$, then, up to isotopy of this parametrization, $\delta_0(\alpha)$ maps to $\rho_i(\alpha)$ under the projection onto the second factor $S^1\times S^1\to S^1: (s,t)\mapsto t$.  Letting $\widehat{X}_i$ denote the 3-manifold obtained by Dehn filling along $\mu_i(\alpha)$, $i=1,2$, it follows that $\phi(\alpha)$ is to the left of $\alpha$ in the associated open book of $\widehat{X}_i$ if and only if $\phi(\alpha)$ is to the right of $\alpha$ in the associated open book of $\widehat{X}_j$, for $j\ne i$.

Next consider a   properly embedded  oriented essential arc   $\beta$ such that the arcs $\alpha,\beta,\phi(\alpha)$, and $\phi(\beta)$ are pairwise nonisotopic as unoriented arcs and pairwise tight. Orient the arc  $\beta$. Up to symmetry, including interchanging the labelings $\rho_1$ and $\rho_2$, there are three possibilities:
\begin{enumerate}

\item $\{\mu_1(\alpha),\mu_2(\alpha)\}\cap \{\mu_1(\beta),\mu_2(\beta)\}=\{\mu_1(\alpha)\}=\{\mu_1(\beta)\}$;

\item $\rho_1(\alpha)\cap \rho_1(\beta)=\emptyset$ but $\{\mu_1(\alpha),\mu_2(\alpha)\} = \{\mu_1(\beta),\mu_2(\beta)\}$;

\item $\rho_i(\alpha)\cap \rho_j(\beta)\ne\emptyset$ for all $i, j = 1, 2$ (and, hence, $\{\mu_1(\alpha),\mu_2(\alpha)\} = \{\mu_1(\beta),\mu_2(\beta)\}$).

\end{enumerate}
These are illustrated in Figure~\ref{fig: arcpairs}.

\begin{figure}[ht]
\labellist
\small
\pinlabel $\rho_1(\alpha)$ at 160 408
\pinlabel $\rho_1(\beta)$ at 234 408
\pinlabel $\rho_1(\alpha)$ at 363 408
\pinlabel $\rho_1(\beta)$ at 442 408
\pinlabel { (1)} at 180 284
\pinlabel { (2)} at 395 284
\pinlabel { (3)} at 600 284
\pinlabel $\lambda$ at 266 350
\pinlabel $\lambda$ at 478 350
\pinlabel $\lambda$ at 678 350
\endlabellist
\begin{center}
\includegraphics[scale=.45]{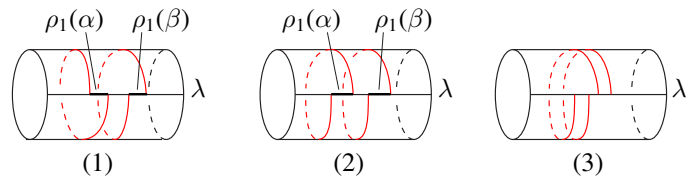}
\end{center}
\caption{Possible relative positions of the transition arcs $\delta_0(\alpha)$ and $\delta_0(\beta)$.}
\label{fig: arcpairs}
\end{figure}

If (1) holds for some choice of $(\alpha,\beta)$, let $\mu=\mu_1(\alpha)$ be the common meridian, and set $\delta'_0(\alpha)= \rho_1(\alpha)$.  In this case, $\phi$ realizes $\mu$ as a closed orbit and $\mu\in \{\mu_1(\gamma),\mu_2(\gamma)\}$ for all choices of $\gamma$. If (2) holds for two choices $(\alpha,\beta)$ and $(\alpha', \beta')$ such that $\mu_1(\alpha)=\mu_2(\alpha')$, so $\mu_1(\alpha) \neq \mu_1(\alpha')$,  we note in passing that $\phi|_{\partial F}$ has a periodic point of order two. If (3) holds for all $(\alpha,\beta)$, we note that $\phi$ is isotopic to a periodic homeomorphism of order two.  (If not, consider an arc $\alpha$ such that $\phi^{-1}(\alpha)$ is not isotopic to $\phi(\alpha)$ and an arc $\beta$ with $\beta(0)$ in the component of $\partial F \setminus \{\phi^{-1}(\alpha(0)), \phi(\alpha(0))\}$ that does not contain $\alpha(0)$.  One of $(\alpha, \beta)$ or $(\phi(\alpha),\beta)$ fails to satisfy (3).)   In these cases, we defer the choice of $\mu$ until after Corollary \ref{correct meridian}.  Otherwise, set $\mu=\mu_1(\alpha)$, as well.

\begin{thm}\label{oldrachel2} \cite{Rfib2}        If (1) holds for some choice of $(\alpha,\beta)$, there are co-oriented taut foliations that strongly realize all slopes   except possibly $\mu$. Otherwise, there are co-oriented taut foliations that strongly realize all slopes in the interval of slopes containing $\lambda$ that lie strictly between $\mu_1(\alpha)$ and $\mu_2(\alpha)$ for some, and hence all, choices of tight $\alpha$.
\end{thm}

If $X$ is the complement of a fibered knot $\kappa \subset S^3$, and $\alpha$ is a tight essential arc in $F$, then in terms of the standard parametrization of $\partial N(\kappa)$, $\mu_1(\alpha)=1/n$ and $\mu_2(\alpha)=1/(n+1)$, for some $n\in \mathbb Z$.  If $n\notin \{-1,0\}$, then by Theorem \ref{oldrachel2}, $1/0$ would be strongly realized by a taut foliation, an impossibility;  hence, $1/0 \in \{\mu_1(\alpha), \mu_2(\alpha)\}$.  In  fact, the proof of Theorem 4.5 of \cite{KR} implies more:

\begin{cor} \label{correct meridian} If $X$ is the complement of a fibered knot $\kappa \subset S^3$, one of $\mu_1(\alpha)$ or $\mu_2(\alpha)$ is the standard meridian for some, and hence any, choice of tight $\alpha$; moreover, $\mu$ as defined as above is the standard meridian.
\end{cor}

Finally, in the cases for which $\mu$ has not been defined above, we define it as follows:  if $X$ is the complement of a fibered knot in $S^3$, choose $\mu$ to be the standard meridian;  otherwise, with reference to the observation above, choose $\mu \in \{\mu_1(\alpha), \mu_2(\alpha)\}$ so that $\phi(\alpha)$ is to the right of $\alpha$ in the $\mu$-Dehn  filling of $F\times [0,1]/\sim$.  Set $\delta'_1(\alpha)=\delta'_0(-\alpha)$. With these definitions we have $\mu=\delta_i(\alpha)\cup\delta'_i(\alpha)$ for all tight $\alpha$ and for $i=1,2$.

\begin{notation} \label{notation: meridian}
Parting with previous notation, henceforth let $\mu_i(\alpha) =\delta_i(\alpha)\cup\delta'_i(\alpha)$, $i = 0,1$. 
\end{notation}

\begin{definition}
Each  $\delta_i(\alpha)$ is a \emph{transition arc}, and  each $\delta'_i(\alpha)$ is the \emph{meridian complement} of   $\delta_i(\alpha)$. 
\end{definition}

In summary, parametrize  $\partial M$  as $S^1\times S^1$ so that  $\{1\}\times S^1$ represents the distinguished meridian $\mu$, and $S^1\times \{1\}$ represents the longitude $\lambda$, the isotopy class of $\partial F$. 
If we consider   $D(\alpha)$, for some tight $\alpha$, and focus on a neighourhood of  one of $\delta_0(\alpha)$ or $\delta_1(\alpha)$,  then either $\alpha=\phi(\alpha)$ as oriented arcs, or we see one of the models shown in Figure~\ref{RVLV}. Reversing the orientation convention found in \cite{Rfib1,Rfib2},  we call the first transition arc \emph{positive}  and the second \emph{negative}.  Notice that reversing the orientation of $\kappa$ reverses the orientation on $F$, and vice versa; so,  the sign of a transition arc is independent of the initial choice of orientation on $\kappa$.  Either the endpoints of $\alpha$ separate those of $\phi(\alpha)$ or they do not; up to symmetry, the possibilities when $\alpha$ is not isotopic to $\phi(\alpha)$ (as unoriented arcs) are listed in Figure~\ref{endpointspossibilities}.   

\begin{remark}
We observe that the open book decomposition in $\widehat{X}(\mu)$ associated to $(F, \phi)$ is \emph{right-veering} (respectively, \emph{left-veering}) \cite{HKM1} if,    for every properly embedded tight oriented arc $\alpha \in F$, either $\alpha=\phi(\alpha)$ or the transition arc $\delta_0(\alpha)$ is positive (respectively, negative).
\end{remark}

\begin{figure}[ht]
\labellist
\small
\pinlabel $\kappa$ at 84 433
\pinlabel $\kappa$ at 430 433
\pinlabel \textcolor{Red}{\footnotesize{$\mu$}} at 300 447
\pinlabel \textcolor{Red}{\footnotesize{$\mu$}} at 643 447
\pinlabel \textcolor{Red}{\footnotesize{$\phi(\alpha)$}} at 208 290
\pinlabel \textcolor{Red}{\footnotesize{$\alpha$}} at 278 290
\pinlabel \textcolor{Red}{\footnotesize{$\phi(\alpha)$}} at 580 290
\pinlabel \textcolor{Red}{\footnotesize{$\alpha$}} at 510 290
\pinlabel $+$ at 154 370
\pinlabel $+$ at 500 370
\pinlabel {Positive transition} [Bl] at 111 220
\pinlabel {Negative transition} [Bl] at 451 220
\endlabellist
\begin{center}
\includegraphics[scale = .4]{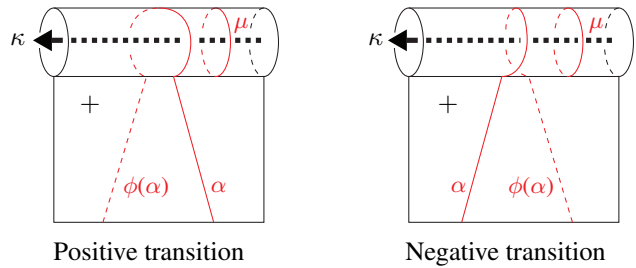}
\end{center}
\caption{Positive and  negative  transition models.}\label{RVLV}
\end{figure} 

\begin{figure}[ht]
\labellist
\small
\pinlabel $+$ at 141 480
\pinlabel $+$ at 450 480
\pinlabel $+$ at 697 480
\pinlabel $+$ at 385 190
\pinlabel a at 385 340
\pinlabel b at 385 50
\endlabellist
\begin{center}
\includegraphics[scale=.4]{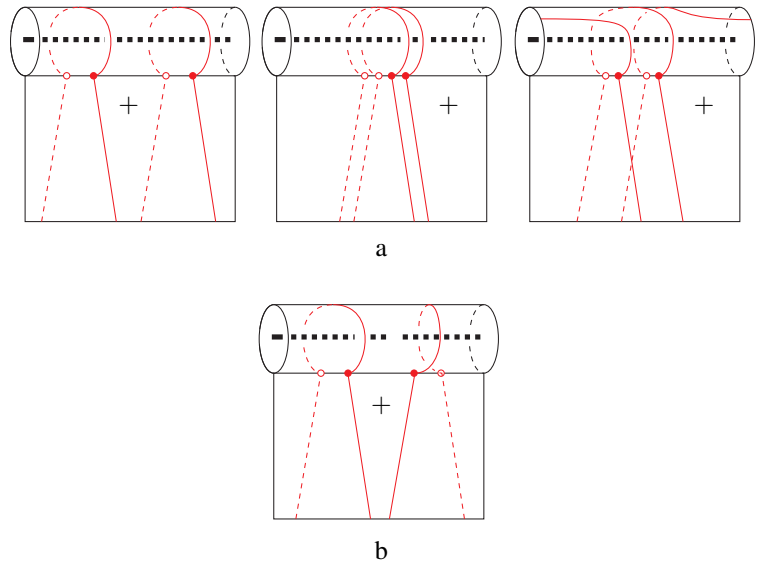}
\end{center}
\caption{Possibilities for the transition arc pair:  (a) positive transitions (b) positive and negative transition.}\label{endpointspossibilities}
\end{figure} 

\begin{definition}
Given any metric space $X$ and any subset $A \subset X$, the \emph{closed complement} of $A$ in $X$,  denoted $X|_A$,  is the metric completion of $X\setminus A$.
\end{definition}

\begin{remark}
Intuitively, $X|_A$ amounts to cutting $X$ open along $A$.  (Although other authors have used the notation $X_A$, we feel the inclusion of a vertical slash evokes the notion of ``cutting.")   In particular, we will consider the closed complement of a curve in a surface and of a surface or, more generally, a lamination \cite{GO}, in a $3$-manifold, with respect to the path metric inherited from a Riemannian metric.  For example, if $F$ is a fiber of a fibered knot $\kappa$, then $M|_F$ is homeomorphic to $F\times [0,1]$. \end{remark}

 \begin{definition} \label{tightproddisk}
Let $M$ be a 3-manifold with nonempty boundary, and let $S$ be an oriented surface with nonempty boundary properly embedded in $M$.  Label the two copies of $S$ in $(M|_S)$ by $S_-$ and $S_+$.  A \emph{product} disk is  an immersed disk in $M$ whose pre-image under the quotient map $M|_S \to M$ is properly embedded in $(M|_S)$ with boundary consisting   of two essential arcs in $\partial M$, and two essential arcs in $S$, one contained in $S_-$ and one in $S_+$.   A product disk is \emph{tight} if the two arcs of its boundary in $S$ are tight.   
\end{definition}

In particular, the disk $D(\alpha)$ of Notation \ref{defntransitionarc1} is a tight product disk whenever $\alpha$ is tight.

\subsection{Laminations and foliations}

Roughly speaking, a  codimension-one foliation $\mathcal F$ of a 3-manifold $M$ is a disjoint union of surfaces injectively immersed in $M$ such that $(M,\mathcal F)$ looks locally like $(\mathbb R^3,\mathbb R^2 \times \mathbb  R)$. More precisely, we have the following definition.

\begin{definition}
Let $M$ be a closed $C^{\infty}$ 3-manifold. 
A \textit{codimension one foliation} $\mathcal F$ of (or in) $M$ is   a decomposition of $M$ into  disjoint  connected  surfaces $L_i$, called the \textit{leaves of $\mathcal F$},   such that  $(M,\mathcal F)$ looks locally like $(\mathbb R^3,\mathbb R^2 \times \mathbb R)$. More precisely:
\begin{enumerate}
\item $\cup_i L_i = M$, and
\item  there exists an atlas $\mathcal A$ on $M$  with respect to which $\mathcal F$  respects  the following local product structure: 
\begin{itemize}
\item for every $p\in M$, there exists a coordinate chart $(U,(x,y,z))$ in $\mathcal A$ about $p$ such that $U$ is homeomorphic to $\mathbb R^3$ and the restriction of $\mathcal F$ to $U$ is the union of
planes given by $z = $ constant. 
\end{itemize}
\end{enumerate}
A foliation is \emph{co-oriented} if the leaves admit co-orientations that are locally compatible.  

We also consider foliations $\mathcal F$ in compact smooth 3-manifolds with nonempty boundary, restricting attention to the case that $\partial M$ is a nonempty union of tori and $\mathcal F$ intersects $\partial M$ everywhere transversely. In this case, at boundary points of $M$, 
$(M,\mathcal F)$ looks locally like horizontal closed half planes $[0,\infty)\times\mathbb R$.
\end{definition}
Calegari \cite{calegari} proved that any   foliation has an isotopy representative that is $C^{\infty,0}$; in particular,  such that $T\mathcal F$ is defined and continuous,  and leaves of $\mathcal F$ are smoothly immersed.  A foliation is  \textit{taut} \cite{Gab,KRtaut} if for every $p\in M$ there exists a 1-submanifold that contains $p$ and is  everywhere transverse to $\mathcal F$.  The foliations constructed in this paper will have only noncompact leaves; hence they  have an isotopy representative that is taut \cite{Haefliger,KRtaut}. 

Recall that a subset of $M$  is \emph{$\mathcal F$-saturated} if it is a union of leaves of $\mathcal F$. A \emph{minimal set} of $\mathcal F$ is a closed $\mathcal F$-saturated subset of $M$ that doesn't properly contain a nonempty closed $\mathcal F$-saturated subset. The foliations constructed in this paper  contain a unique minimal set, and this minimal set is disjoint from $N(\kappa)$.

A \emph{lamination} $\mathcal L$ is a  decomposition of a closed subset of $M$ into a union of injectively immersed surfaces, called the \emph{leaves} of $\mathcal L$,  such that  $(M,\mathcal L)$ looks locally like $(\mathbb R^3,\mathbb R^2 \times C)$, where $C$ is a closed subset of $\mathbb R$.   Properly embedded compact surfaces, foliations, and $\mathcal F$-saturated closed subsets of $M$, such as minimal sets of foliations, are all key examples of laminations.  All laminations that arise in this paper are $\mathcal F$-saturated closed subsets of $M$ for some foliation $\mathcal{F}$. 

A lamination \emph{strongly realizes}     the slope $r\in S^1$ if it meets $\partial N(\kappa)$   transversely in a lamination consisting of  consistently oriented curves of slope $r$. When $r$ is rational,  these curves are closed, and it makes sense to talk about the manifold $\widehat{M}_r$ obtained by Dehn surgery along $\kappa$ by slope $r$. If a lamination $\mathcal L$ strongly realizes slope $r$, then it extends to a lamination $\widehat{\mathcal L}$ in $\widehat{M}_r$ by capping off each boundary curve  with disk. Moreover, if $\mathcal L$ is a co-oriented taut foliation, then so is $\widehat{\mathcal L}$.

\section{Any  composite knot with a persistently foliar summand is persistently foliar. } \label{lifeeasy}

Before discussing connected sums of fibered knots, we prove the useful fact that strong realization of a slope for a knot in $S^3$ extends to any connected sum with that knot.

\begin{prop} \label{anyproblem}
Suppose $\kappa=\kappa_1\#\kappa_2$  is a connected sum of knots in $S^3$. If the slope $m$ along $\kappa_1$ is strongly realized by a co-oriented taut foliation, then so is the slope $m$ along $\kappa$. 
\end{prop}

\begin{proof} Suppose  $X_{\kappa_1}$   contains a co-oriented taut foliation $\mathcal F_1$ that strongly realizes slope $m$. By    \cite{Gab3}, there is a co-oriented taut foliation $\mathcal F_2$ in $X_{\kappa_2}$ that strongly realizes the longitudinal boundary slope.  We  describe how to form a ``connected sum" of these two foliations  to produce a co-oriented taut foliation that strongly realizes slope $m$ in $X_{\kappa}$.

 Let $P$ denote a summing sphere for this connected sum, and set $A := P\cap (S^3\setminus \Int N(\kappa)) = P \cap X_\kappa.$   Observe that   for each $i = 1, 2$,  $\partial X_{\kappa_i}$ is the union of two annuli, one of which is $A$.     Viewing $A$ as $S^1 \times I$, we may arrange that   each foliation $\mathcal{F}_i$ intersects $A$ with leaves $\theta \times I$, $\theta \in S^1$. It is easy to check that, gluing $\mathcal F_1$ to $\mathcal F_2$ along $A$, we obtain a taut foliation $\mathcal{F}$ in $X_\kappa$ that realizes slope $m$.   Choosing compatible co-orientations on the foliations $\mathcal F_1$ and $\mathcal F_2$ yields a co-orientation for $\mathcal{F}$.    
\end{proof}

\begin{cor}\label{consum} Suppose $\kappa=\kappa_1\#\cdots \#\kappa_n$ is a connected sum of  knots. If at least one of the $\kappa_i$ is persistently foliar, then so is $\kappa$.\qed 
\end{cor}

\section{Spines, Train Tracks and Branched surfaces} \label{toodles}

 In this paper we construct foliations by first constructing a spine, which we then smooth to a branched surface that carries a foliation (more precisely, a lamination that extends to a foliation).  We restrict attention to the case that any intersections of a  spine or a branched surface with $\partial M$ are transverse; hence the intersection of a branched surface with the boundary of a 3-manifold is a train track.  The curves carried by this train track play an important role in our analysis.
 
A \emph{train track} is a space locally modeled on one of the spaces of Figure~\ref{traintrackmodels}. An \emph{$I$-fibered neighbourhood} of a train track $\tau$ is a regular neighborhood $N(\tau)$ foliated (as a 2-manifold with corners) by interval fibers that intersect $\tau$ transversely, as locally modeled by the spaces in Figure \ref{traintracknbhd}. 
 
\begin{figure}[ht]
\labellist
\small
\pinlabel {1-manifold} [B] at 195 135
\pinlabel {neighborhood} [B] at 195 95
\pinlabel {double point} [B] at 570 135
\pinlabel {neighborhood} [B] at 570 95
\endlabellist
\begin{center}
\includegraphics[scale=.3]{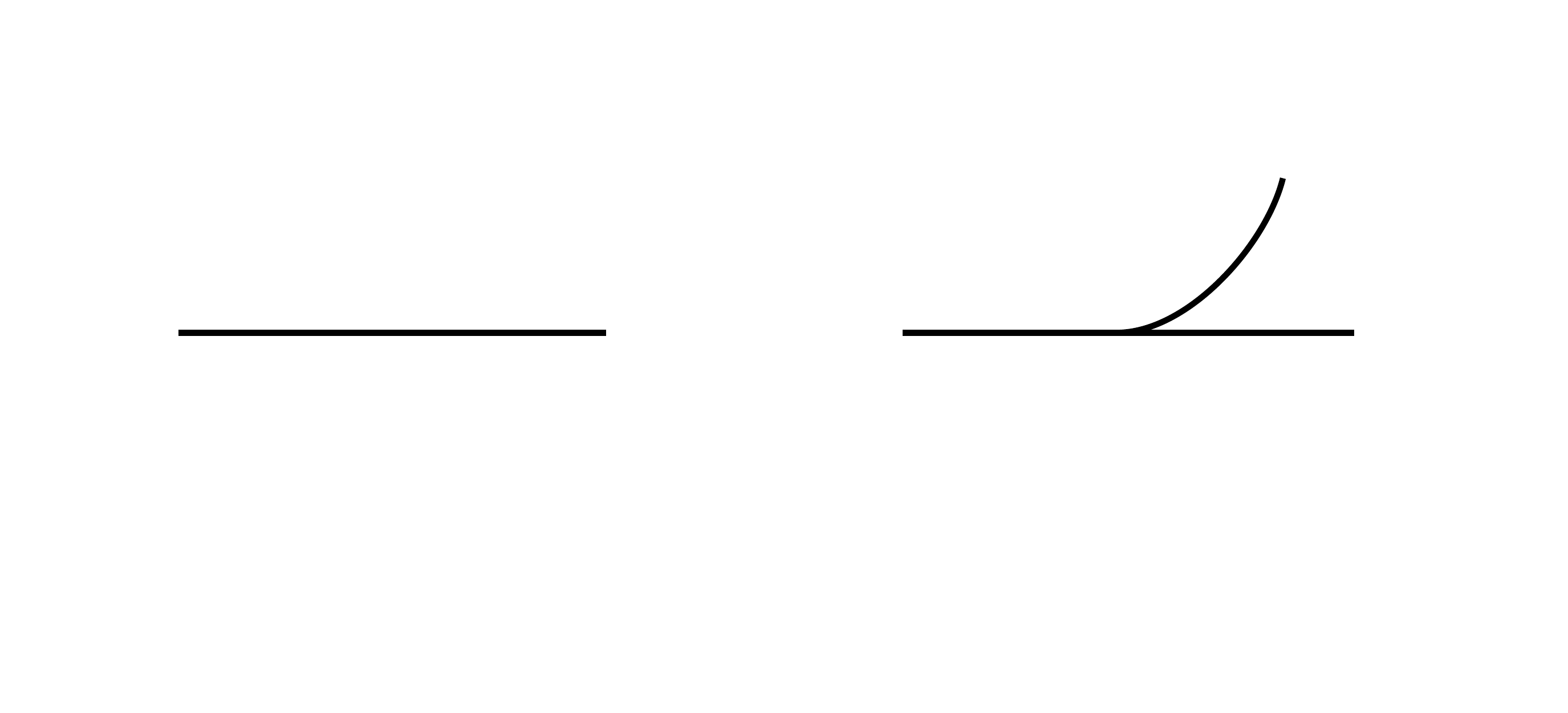}
\end{center}
\caption{Local models of a train track.}\label{traintrackmodels}
\end{figure}
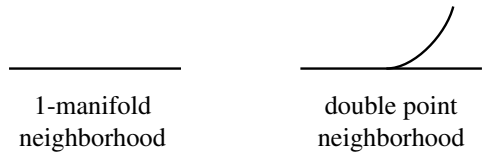 

\begin{figure}[ht]
\labellist
\small
\pinlabel {1-manifold} [B] at 170 100
\pinlabel {neighborhood} [B] at 170 70
\pinlabel {double point} [B] at 552 100
\pinlabel {neighborhood} [B] at 552 70
\endlabellist
\begin{center}
\includegraphics[scale=.4]{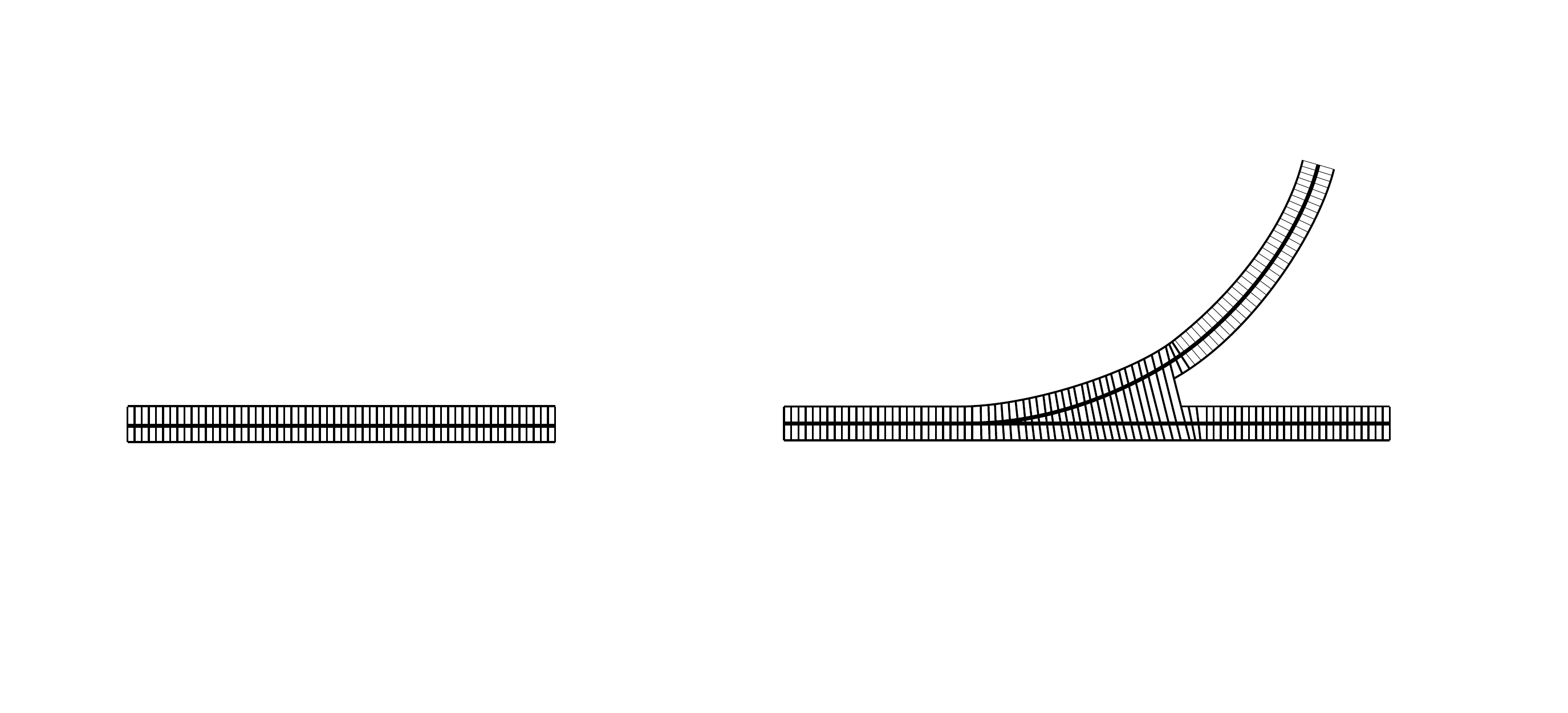}
\end{center}
\caption{Local models of an $I$-fibered neighbourhood of a train track.}\label{traintracknbhd}
\end{figure} 

A {\it standard spine} \cite{C} is a space $\Sigma$ locally modeled on one of the spaces of Figure~\ref{spine}.   A standard spine with boundary has the additional local models shown in Figure~\ref{spineboundary}. The {\it critical locus}  $\Gamma$  of $\Sigma$ is the 1-complex of points of $\Sigma$ where the spine is not locally a manifold.  The critical locus is a stratified space (graph) consisting of triple points $\Gamma^0$ and arcs of double points $\Gamma^1 = \Gamma \setminus \Gamma^0$.

\begin{definition}
 The components of $\Sigma|_{\Gamma}$ are called the {\it sectors} of $\Sigma$. 
\end{definition}

\begin{figure}[ht]
\labellist
\small
\pinlabel surface [B] at 170 200
\pinlabel neighborhood [B] at 170 170
\pinlabel {double point} [B] at 395 200
\pinlabel neighborhood [B] at 395 170
\pinlabel {triple point} [B] at 615 200
\pinlabel neighborhood [B] at 615 170
\endlabellist
\begin{center}
\includegraphics[scale = .4]{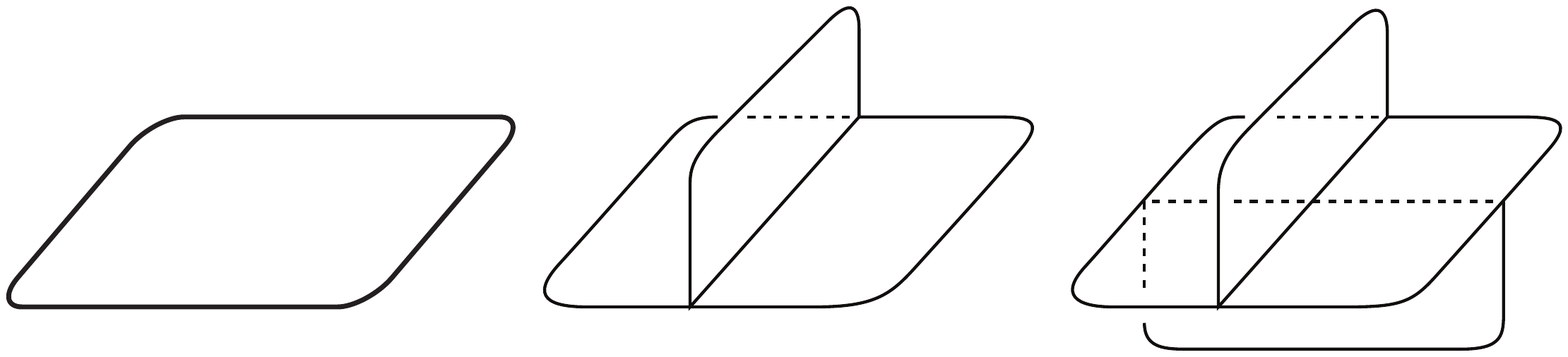}
\end{center}
\caption{Local models of a standard spine at interior points.}\label{spine}
\end{figure} 

\begin{figure}[ht]
\labellist
\small
\pinlabel surface [B] at 235 215
\pinlabel neighborhood [B] at 235 185
\pinlabel {double point} [B] at 535 215
\pinlabel neighborhood [B] at 535 185
\pinlabel {at boundary} [B] at 235 155
\pinlabel {at boundary} [B] at 535 155
\pinlabel {$\partial \Sigma$} at 375 210
\endlabellist
\begin{center}
\includegraphics[scale=.4]{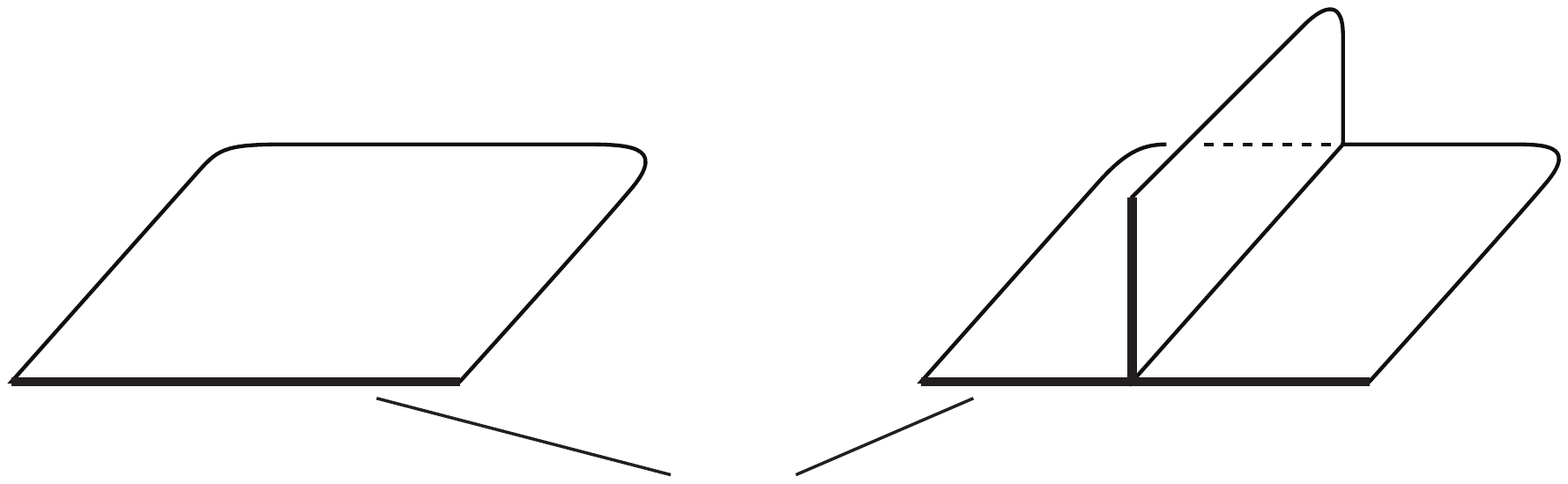}
\end{center}
\caption{Local models of a standard spine at boundary points.}\label{spineboundary}
\end{figure} 

A {\it branched surface (with boundary)} (\cite{W}; see also \cite{O1,O2}) is a space $B$ locally modeled on the spaces of Figure~\ref{brsurf} (along with those in Figure~\ref{brsurfboundary}); that is, $B$ is homeomorphic to a spine, with the additional structure of a well-defined tangent plane at each point. The {\it branching locus} $\Gamma$ of $B$ is the 1-complex of points of $B$ where $B$ is not locally a manifold; such points are called \emph{branching points}.  The branching locus is a stratified space (graph) consisting of triple points $\Gamma^0$ and arcs of double points $\Gamma^1 = \Gamma \setminus \Gamma^0$. The components of $B|_\Gamma$ are called the {\it sectors} of $B$.  

\begin{figure}[ht]
\labellist
\small
\pinlabel surface [B] at 170 210
\pinlabel neighborhood [B] at 170 180
\pinlabel {double point} [B] at 390 210
\pinlabel neighborhood [B] at 390 180
\pinlabel {triple point} [B] at 615 210
\pinlabel neighborhood [B] at 615 180
\endlabellist
\begin{center}
\includegraphics[scale=.4]{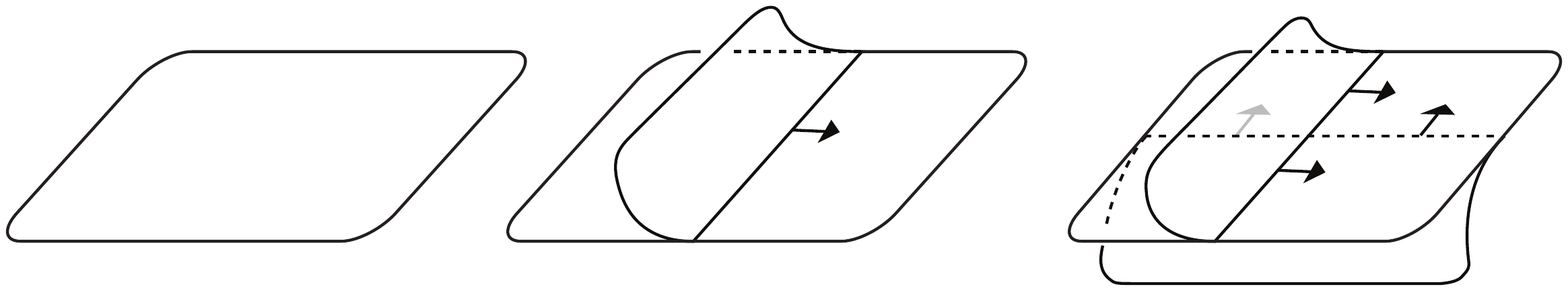}
\end{center}
\caption{Local model of a branched surface at interior points.}\label{brsurf}
\end{figure}

\begin{figure}[ht]
\labellist
\small
\pinlabel surface [B] at 215 225
\pinlabel neighborhood [B] at 215 195
\pinlabel {double point} [B] at 535 225
\pinlabel neighborhood [B] at 535 195
\pinlabel {at boundary} [B] at 215 165
\pinlabel {at boundary} [B] at 535 165
\pinlabel {$\partial \Sigma$} at 355 220
\endlabellist
\begin{center}
\includegraphics[scale=.4]{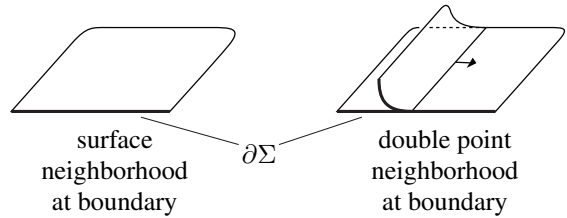}
\end{center}
\caption{Local model of a branched surface at boundary points.}\label{brsurfboundary}
\end{figure}

An \emph{$I$-fibered neighborhood} of a branched surface $B$  in a $3$-manifold $M$ is a regular neighborhood $N(B)$ foliated by interval fibers that intersect $B$ transversely, as locally modeled by the spaces in Figure \ref{figure: I-bundle neighborhood} at interior points;  if the ambient manifold $M$ has non-empty boundary, all spines and branched surfaces are assumed to be properly embedded, with $N(B)\cap\partial M$ a union (possibly empty) of I-fibers. The surface $\partial N(B) \setminus \partial M$ is a union of two subsurfaces,  $\partial_v N(B)$ and  $\partial_h N(B)$, where $\partial_v N(B)$, the \emph{vertical boundary},  is a union of sub-arcs of I-fibers, and $\partial_h N(B)$,  the \emph{horizontal boundary}, is everywhere transverse to the I-fibers.  

\begin{figure}[ht]
\labellist
\small
\pinlabel {Horizontal Boundary} [Bl] at 270 345
\pinlabel {Vertical boundary} [Br] at 270 310
\pinlabel surface [B] at 120 105
\pinlabel neighborhood [B] at 120 75
\pinlabel {double point} [B] at 390 105
\pinlabel neighborhood [B] at 390 75
\pinlabel {triple point} [B] at 635 105
\pinlabel neighborhood [B] at 635 75
\endlabellist
\begin{center}
\includegraphics[scale=.4]{ibundleneighborhood}
\end{center}
\caption{Local models for $N(B)$ (at interior points).}  \label{figure: I-bundle neighborhood}
\end{figure}

Let $\pi$ be the retraction of $N(B)$ onto the quotient space obtained by collapsing each fiber to a point.  The branched surface $B$ is obtained, modulo a small isotopy, as the image of $N(B)$ under this retraction.  We will freely identify $B$ with this image \cite{O1} and the core of each component of vertical boundary with its image in $\Gamma$.   Double points of the branching locus are \emph{cusps} with \emph{cusp direction} pointing inward from the vertical boundary if $B$ is viewed as the quotient of $N(B)$ obtained by collapsing the vertical fibers to points.  Cusp directions will be indicated by arrows, as in Figures \ref{brsurf}  and \ref{brsurfboundary}.

A branched surface $B$ is \emph{co-oriented} if the one-dimensional foliation of $N(B)$ is oriented. When $\partial M$ is a union of tori, and  co-oriented $B$ is transverse to $\partial M$, the regions $$(\partial M\setminus\Int N(B),\partial_v N(B)\cap \partial M)$$ are products, and we use the notation $$\partial'_v N(B)=\partial_v N(B)\cup  (\partial M\setminus\Int N(B)).$$ Following Gabai \cite{Gab, Gab2, Gab3}, we call each component of $\partial'_v N(B)$ a \emph{suture} and refer to the pair $\left(M \setminus \Int N(B), \partial'_v N(B)\right)$ as a \emph{sutured manifold}.  Note that each component  of $\partial'_v(N(B))$ is an annulus or a torus.  

A surface is said to be \emph{carried by} $B$ if it is contained in $N(B)$ and is everywhere transverse to the one-dimensional foliation of $N(B)$. A surface is said to be \emph{fully carried by} $B$ if it carried by $B$ and has nonempty intersection with every I-fiber of $N(B)$. A lamination $\mathcal L$ is  carried by $B$ if each leaf of $\mathcal L$ is carried by $B$, and  fully carried if, in addition, each I-fiber of $N(B)$ has nonempty intersection with some leaf of $\mathcal L$. 

Similarly,   a  1-manifold, or union of 1-manifolds, is said to be \emph{carried by}  a train track $\tau$ if it is contained in some I-fibered neighbourhood $N(\tau)$, everywhere transverse to the one-dimensional foliation of $N(\tau)$. A 1-manifold, or union of 1-manifolds, is said to be \emph{fully carried by} $\tau$ if it  is carried by $\tau$ and has nonempty intersection with every I-fiber of $N(\tau)$.

\begin{figure}[ht]
\labellist
\small
\pinlabel {\tiny $\pm$} at 243 360
\pinlabel {\tiny $\pm$} at 295 327
\pinlabel {\tiny $\pm$} at 285 270
\endlabellist
\begin{center}
\includegraphics[scale=.5]{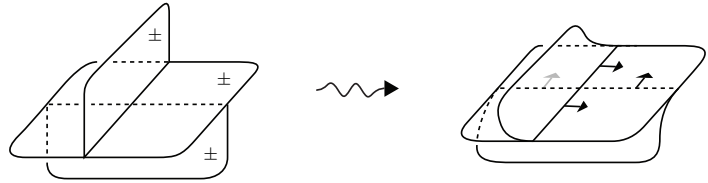}
\end{center}
\caption{Oriented spine to oriented branched surface.}\label{spinetobrsurf} 
\end{figure}

 If a branched surface $B$ is homeomorphic to a spine $\Sigma$, we say that $B$ is obtained by \emph{smoothing} $\Sigma$.     An example is    illustrated in 
Figure~\ref{spinetobrsurf}. We say that a  choice of   co-orientations on the sectors of $\Sigma$ is \emph{compatible} if there is a smoothing of $\Sigma$ to a co-oriented branched surface $B$ that preserves the co-orientations on sectors; in this case, we call this  smoothing the \emph{smoothing determined by the co-orientations}. Examples are    illustrated in 
Figure~\ref{smoothingexample} and Figure~\ref{smoothingexample2}. Branched surfaces  $B^G(\alpha)=\langle F;D(\alpha) \rangle$ as described   below  play a key role in this paper. We use the superscript $G$  because Lemma~\ref{BGabai} describes a special case of Gabai's Construction~4.16 of \cite{Gab3} applied in the context of \cite{Gabfib}.

\begin{notation}
Let $X$ be a fibered 3-manifold, with compact fiber $F$ and monodromy $\phi$;  assume $\partial F$ is connected and non-empty.  (For current purposes, of course, we focus on the complement $X_\kappa$ of a fibered knot $\kappa$.)  Let $\alpha$ be a tight arc in $F$ such that $\alpha\ne\phi(\alpha)$ (as unoriented arcs).  Set  $\Sigma=F\cup D(\alpha)$.   Any choice of orientations on the surfaces $F$ and   $D(\alpha)$  determines a unique compatible smoothing of $\Sigma$ to a branched surface. We denote this branched surface by   $B^G(\alpha) =\langle F;D(\alpha)\rangle$.   (For purposes of this notation, $F$ and  $D(\alpha)$ are assumed oriented.) 
\end{notation}

\begin{figure}[ht]
\labellist
\small
\pinlabel $\pm$ at 100 534
\pinlabel $\pm$ at 230 479
\pinlabel $\pm$ at 100 424
\pinlabel $\pm$ at 100 298
\pinlabel $\mp$ at 230 243
\pinlabel $\pm$ at 100 188
\pinlabel {\scriptsize $N(B) \cup \partial M$, showing} [Bl] at 562 110
\pinlabel {\scriptsize the sutures, $\partial'_vN(B)$} [Bl] at 562 85
\endlabellist
\begin{center}
\includegraphics[scale=.4]{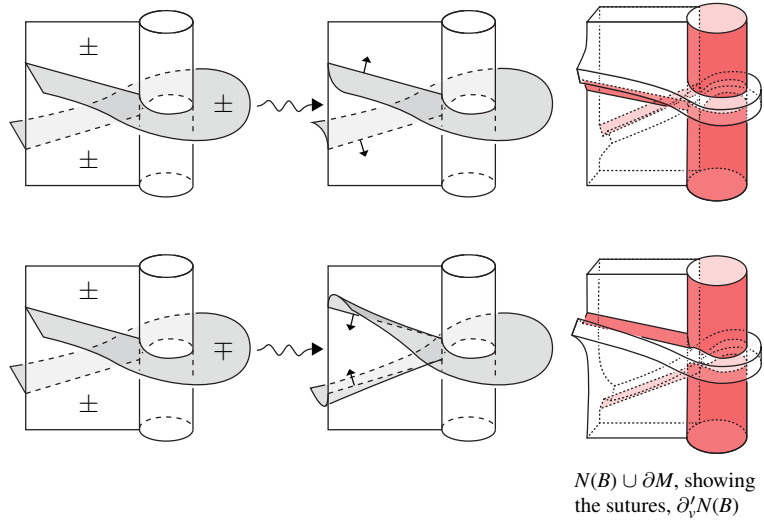}
\end{center}
\caption{$B^G$   near a   transition arc.}\label{smoothingexample} 
\end{figure}

\begin{figure}[ht]
\labellist
\small
\pinlabel $\pm$ at 116 415
\pinlabel $\mp$ at 276 358
\pinlabel $\pm$ at 276 277
\endlabellist
\begin{center}
\includegraphics[scale=.3]{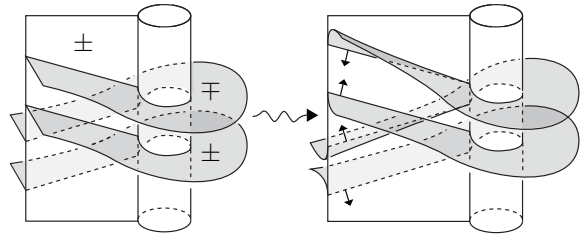}
\end{center}
\caption{$B^G$ near $\partial M$.}\label{smoothingexample2} 
\end{figure}

\begin{lemma}  \label{BGabai}
The complement pair $\left( X\setminus \Int N\left(B^G(\alpha)\right),\partial'_v N\left(B^G(\alpha)\right)\right)$ is homeomorphic to $\left(F|_\alpha \times I,\partial \left(F|_\alpha\right) \times I\right)$.

\end{lemma}

\begin{proof}  The complement of $\Int N(F)$ is homeomorphic to $F\times [0,1]$, and hence is a handlebody of genus twice the genus of $F$, with $\partial'_v N(F) = \partial F\times [0,1]$. Cutting along   $D(\alpha)$ introduces two strips of vertical boundary with cores isotopic to $\alpha$, as illustrated in Figure~\ref{decomprod}, from which it follows   that  the complement of   $N\left(B^G(\alpha)\right)$  is homeomorphic to $F|_\alpha \times I$, with   $\partial'_v N\left(B^G(\alpha)\right) = \partial \left(F|_\alpha\right)\times [0,1]$.
\end{proof}

\begin{figure}[ht]
\labellist
\small
\pinlabel {\footnotesize\textcolor{red}{$\alpha \times [0,1]$}} at 102 558
\pinlabel {$F \times [0,1]$} at 375 490
\pinlabel {$F' \times [0,1]$} at 375 62
\pinlabel {\footnotesize $X_\kappa \setminus N\left(B^G\right)$} at 622 185
\endlabellist
\begin{center}
\includegraphics[scale=.4]{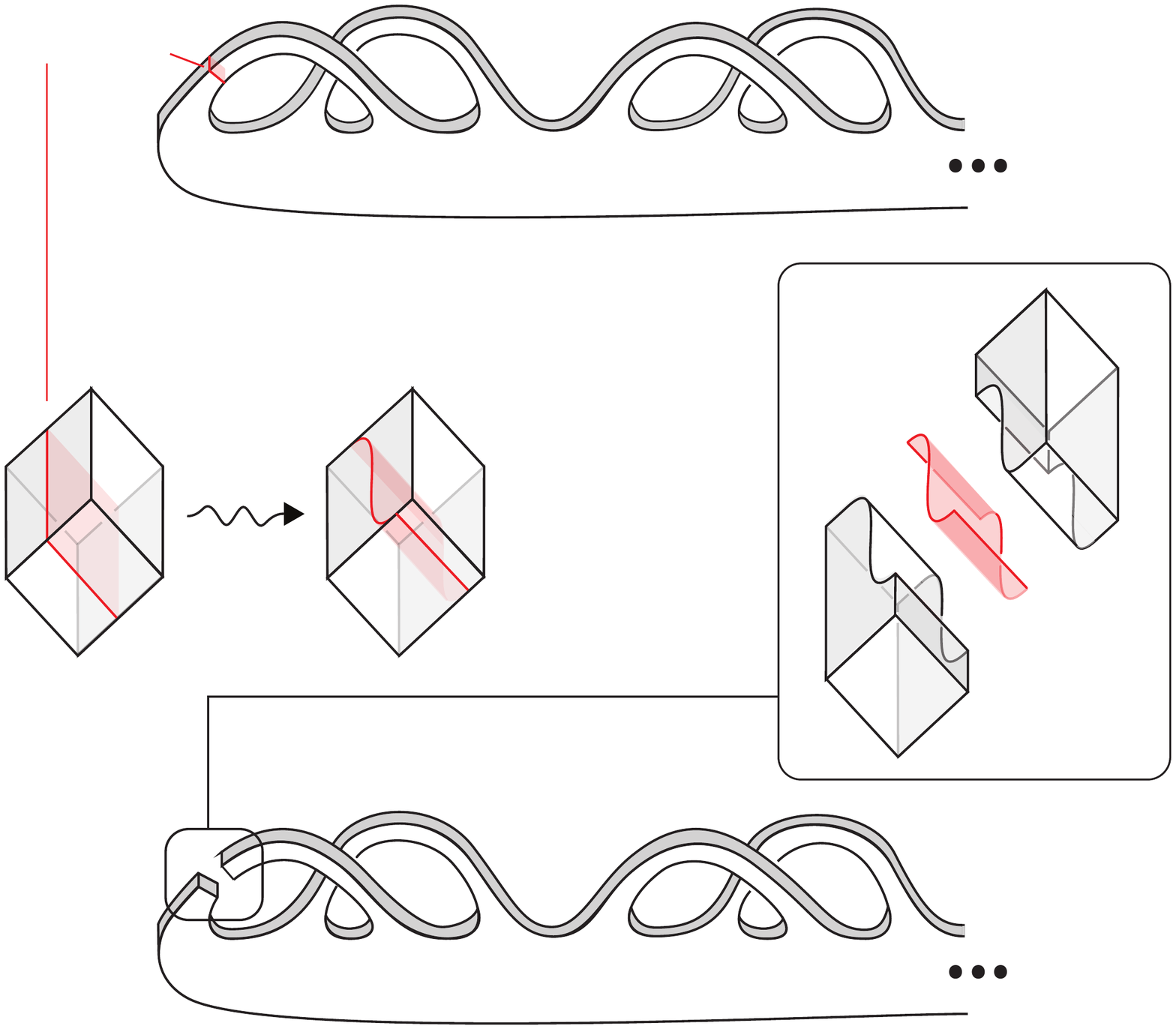}
\end{center}
\caption{$F\times [0,1]$ to  $B^G=\langle F,D(\alpha)\rangle$.}\label{decomprod}
\end{figure}

We next generalize the notation above for a particular family of splittings of the branched surface  $B^G(\alpha)=\langle F; D(\alpha)\rangle$. 
  
\begin{notation}
Suppose $\alpha_1, \alpha_2, \dots, \alpha_{n}=\alpha$ is a sequence of pairwise tight oriented arcs properly embedded in $F$, with $\alpha_i \neq \pm \alpha_{i+1} (\bmod n)$.

Let $D_i= \alpha_{i}\times [\frac{i-1}{n},\frac{i}{n}]$, oriented so that $\partial D_i$ contains  $\alpha_{i}\times \{\frac{i-1}{n}\}$ as a positively oriented subarc. Let $F_i = F \times \{\frac{i}{n}\}$.

We denote by $B^G(\alpha_1, \alpha_2, \dots, \alpha_{n}) = \langle F; (D_i)_{i=1}^n \rangle$ the branched surface obtained, under the identification $(x,1) \sim (\phi(x),0)$, by smoothing the spine $\bigcup_{i=1}^n \left(F_{i-1} \cup D_i\right)$ with co-orientation consistent with the given orientations on the product disks $D_i$ and copies of the fiber $F_i$. Notice that  $B^G(\alpha_1, \alpha_2, \dots, \alpha_{n})= \langle F; (D_i)_{i=1}^n \rangle$  is a splitting of $B^G(\alpha)$ for $n>1$.
\end{notation}

 \subsection{Laminar branched surfaces} \label{laminarsection}  A minimal set of a co-oriented taut foliation  $\mathcal F$ is necessarily essential, and therefore carried by an essential branched surface:

\begin{definition} \cite{GO}\label{defn:GO}
A branched surface $B$ in a  closed 3-manifold $M$ is called an {\it essential} branched surface if it satisfies the following conditions:
\begin{enumerate}
\item  $\partial_h N(B)$ is incompressible in $M\setminus \Int(N(B))$, no component of  $\partial_h N(B)$ is a sphere and $M\setminus \Int(N(B))$ is irreducible.

\item There is no monogon in $M\setminus \Int(N(B))$; i.e., no disk $D\subset M\setminus \Int(N(B))$
with $\partial D = D\cap N(B) = \alpha\cup\beta$, where $\alpha\subset\partial_v N(B)$ is in an interval fiber of $\partial_v N(B)$ and $\beta\subset \partial_h N(B)$

\item There is no Reeb component; i.e., $B$ does not carry a torus that bounds a solid torus in $M$.
\end{enumerate}
\end{definition}

In practice, it can be difficult to determine whether an essential branched surface fully carries a lamination.  In  \cite{Li0,Li}, Li defines the notion of \emph{laminar}, a very useful criterion that is sufficient (although not necessary) to guarantee that an essential branched surface fully carries a lamination. We recall the necessary definitions here.

\begin{definition} \cite{Li0,Li} Let $B$ be a branched surface in a 3-manifold $M$.   A  {\it sink disk}  is a disk branch sector $D$  of 
$B$  for which the   cusp direction of each component of  $\Gamma^1\cap \overline{D}$  points into $D$  (as shown in Figure \ref{sinkdisk}).
 A {\it half sink disk} is a sink disk which has nonempty intersection with $\partial M$. 
\end{definition}

\begin{figure}[ht]
\begin{center}
\includegraphics[scale=.2]{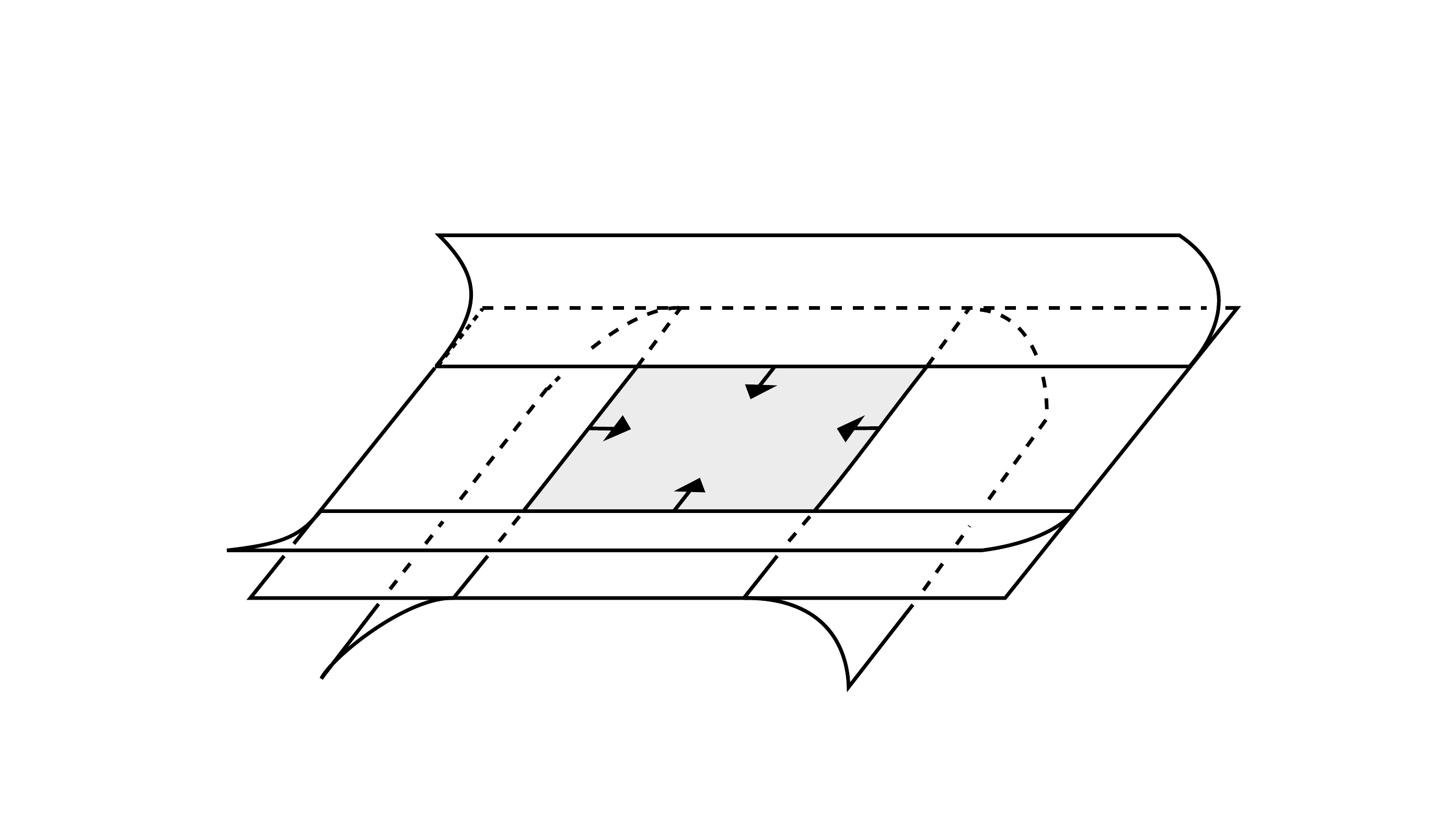}
\end{center}
\caption{A sink disk.}\label{sinkdisk}
\end{figure}

Sink disks  and half sink disks play a key role in Li's notion  of  laminar branched surface. A   sink disk  or half sink disk  $D$ can be  eliminated by splitting $D$ open along a disk in its interior; these trivial splittings must be ruled out:

\begin{definition} (\cite{Li0}, \cite{Li})
Let $D_1$ and $D_2$ be the two disk components of the horizontal boundary of a $D^2 \times I$ region in $M \setminus \Int(N(B))$. If the projection $\pi: N(B) \to B$ restricted to the interior of $D_1 \cup D_2$ is injective, i.e, the intersection of any $I$-fiber of $N(B)$ with $\Int(D_1) \cup \Int(D_2)$ is either empty or a single point, then we say that $\pi(D_1 \cup D_2)$ forms a \emph{trivial bubble} in $B$.
\end{definition}

\begin{definition} (\cite{Li0}, \cite{Li})\label{defn:Li}
An essential branched surface $B$ in a compact 3-manifold $M$ is called  {\it laminar}   if it satisfies the following conditions:
\begin{enumerate}
\item $B$ has no trivial bubbles.
\item $B$ has no sink disk or half sink disk.
\end{enumerate}
\end{definition}

\begin{theorem} (\cite{Li0}, \cite{Li})
Suppose $M$ is a compact and orientable 3-manifold. 
\begin{itemize}
\item[(a)] Every laminar branched surface in $M$ fully carries an essential lamination.
\item[(b)]Any essential lamination in $M$ that is not a lamination by planes is fully carried by a laminar branched surface.
\end{itemize}
\end{theorem}

In general, the branched surface $B^G$ of Lemma~\ref{BGabai} is not laminar. However, $B^G$ admits a splitting to a laminar branched surface. Moreover, this splitting can be chosen so that the boundary train track of the resulting laminar branched surface contains the meridian as subtrack.

\begin{definition}
Let $X$ be an oriented fibered 3-manifold, with monodromy $\phi$ and compact  fiber  $F$;  assume $\partial F$ is connected and nonempty. Let $\alpha$ be a tight arc properly embedded in $F$.  A sequence $\phi(\alpha) = \alpha_0, \alpha_1, \alpha_2, \ldots, \alpha_n = \alpha$ of oriented arcs properly embedded in $F$ is \emph{$\alpha$-sparse} if, for all $j$, 
\begin{enumerate}
\item $\alpha_j\cap\alpha_{j+1}=\emptyset$  and
\item $\alpha_j$ and $\alpha_{j+1}$ are non-isotopic.
\end{enumerate}

\end{definition}

\begin{definition}
For $i = 0, 1$, an $\alpha$-sparse sequence is \emph{$i$-end-effective}, if the endpoints $\alpha_j(i), 0\le j\le n,$ give a monotonic sequence in the interval $\delta'_i(\alpha)$.  An $\alpha$-sparse sequence is \emph{end-effective} if it is both $0$-end-effective and $1$-end-effective.
\end{definition}

\begin{figure}[ht]
\begin{center}
\scalebox{.35}{\includegraphics{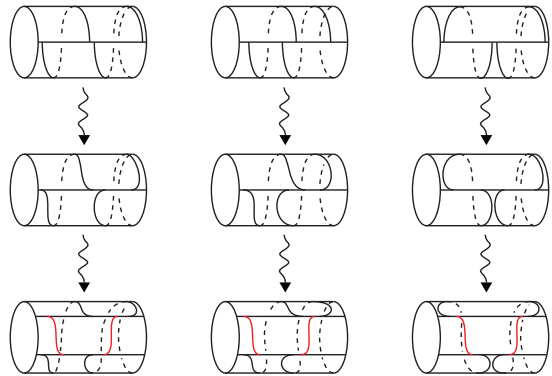}}
\end{center}
\caption{ Sequences that are $i$-end-effective (the example $n=2$), as necessary in the construction of $ B^G(\alpha_1, \alpha_2, \dots, \alpha_{n}) $.}\label{effectiveendpoints}
\end{figure}

\begin{thm} \label{Gcanbemadelaminar}     
Let $X$ be an oriented fibered 3-manifold, with monodromy $\phi$ and compact  fiber  $F$;  assume $\partial F$ is connected and nonempty. Let $\alpha$ be a tight, non-separating, oriented arc properly embedded in $F$ such that $\alpha\ne\phi(\alpha)$ (as oriented arcs).  If both transition arcs are  positive (respectively, negative), then there is a $0$-end-effective (respectively, $1$-end-effective) sequence $\phi(\alpha)=\alpha_0,\alpha_1, \alpha_2, \dots, \alpha_{n} = \alpha$  such that the train track $\tau = B^G(\alpha_1, \alpha_2, \dots, \alpha_{n})\cap \partial M$ contains as subtrack  the meridian, with  the meridional subtrack  necessarily containing   $\delta_0(\alpha)$ (respectively, $\delta_1(\alpha)$).
  
If one transition arc is positive and the other negative, then  there   is an end-effective sequence $\phi(\alpha) = \alpha_0, \alpha_1, \alpha_2, \dots, \alpha_{n}=\alpha$  such the   train track $\tau = B^G(\alpha_1, \alpha_2, \dots, \alpha_{n})\cap \partial M$ contains as subtrack two disjoint copies of  the meridian, with one component of the subtrack containing $\delta_0(\alpha)$ and the other containing $\delta_1(\alpha)$.  
  
In each case,   the resulting branched surface $B^G(\alpha_1, \alpha_2, \dots, \alpha_{n})$  is necessarily laminar.
\end{thm} 

\begin{proof} This is primarily a restatement of results found in \cite{Rfib1,Rfib2}.  There are three possibilities, as illustrated in Figure~\ref{effectiveendpoints}.

When the transitions have a common sign and  $\delta'_0(\alpha)\cap\delta'_1(\alpha)=\emptyset$, this is  the main result of \cite{Rfib1} together with Corollary~6.4 of \cite{Rfib2}. Corollary~6.4 of \cite{Rfib2} is easily modified to allow for the case that  $\delta'_0(\alpha)\cap\delta'_1(\alpha)\ne\emptyset$; a hint is shown in Figure~\ref{overlapcase}. When the transitions are of opposite sign, this is the main construction of \cite{Rfib1} together with Corollary~6.6 of \cite{Rfib2}.  \end{proof}

\begin{figure}[ht]
\labellist
\small
\pinlabel {\footnotesize $F$} [Bl] at 280 185
\pinlabel {\footnotesize $F$} [Bl] at 613 185
\pinlabel {\footnotesize $\partial F$} [tr] at 192 148
\pinlabel {\footnotesize $\partial F$} [tr] at 524 148
\pinlabel or [Bl] at 406 143
\endlabellist
\begin{center}
\includegraphics[scale=.3]{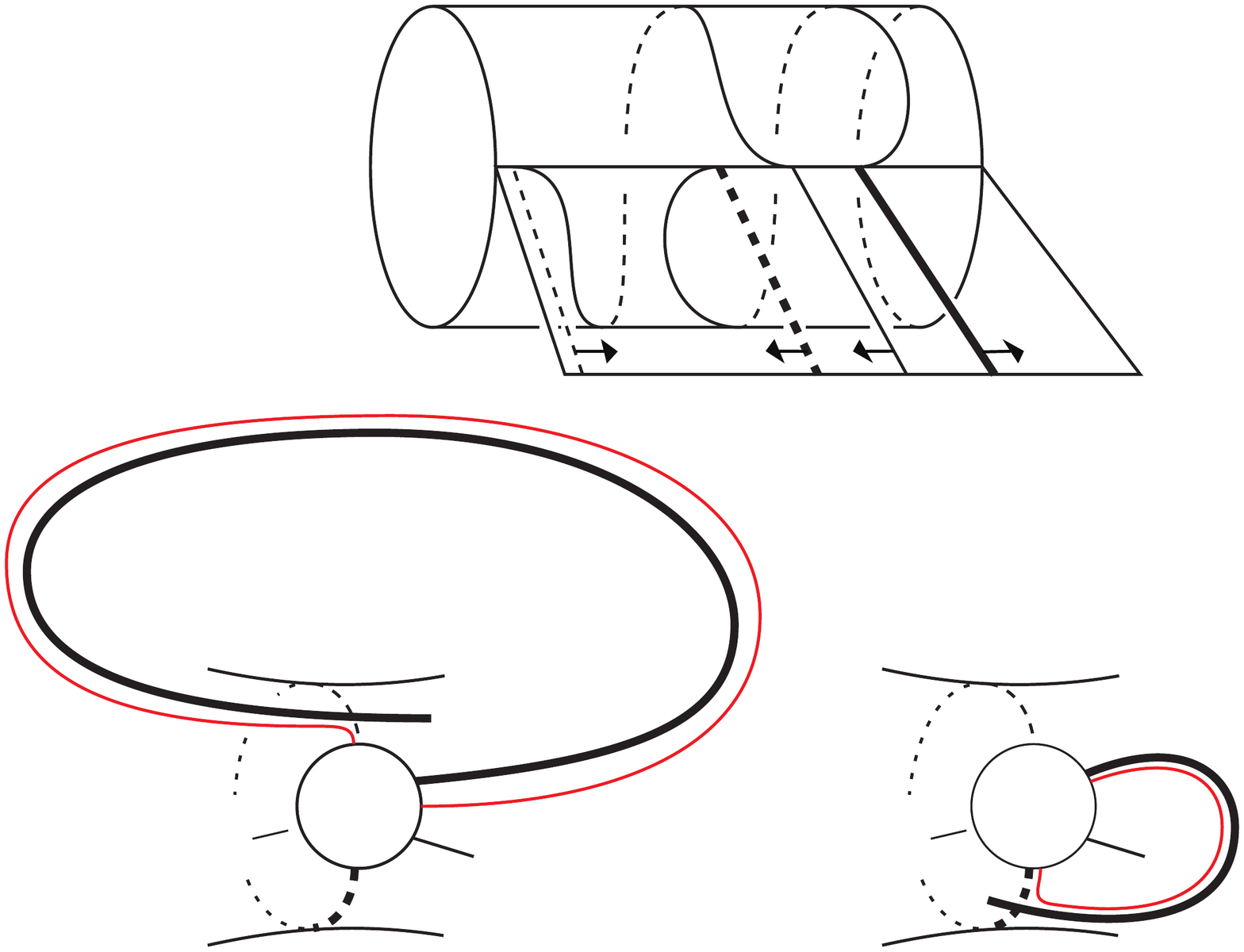}
\end{center}
\caption{}\label{overlapcase}
\end{figure}

In particular, it follows that  $B^G(\alpha_1, \alpha_2, \dots, \alpha_{n})\cap\partial M$ carries all meridians.  (Recall that in this general context, a meridian is defined to be any curve having a single point of minimal transverse intersection with $\partial F$, and we use the definite article and the letter $\mu$ when referring to the distinguished meridian defined in Section \ref{section: prelim}.)  When the transition arcs are of opposite sign, $B^G(\alpha_1, \alpha_2, \dots, \alpha_{n})\cap\partial M$  \emph{fully} carries all meridians except $\mu$.    When both transition arcs have the same sign, $B^G(\alpha_1, \alpha_2, \dots, \alpha_{n})\cap\partial M$  fully carries all meridians except the two, which we call \emph{extremal}, obtained by taking the union of  $\delta_i(\alpha)$ with, respectively,   each of the components of  $\partial F \setminus \delta_i(\alpha)$. When $\delta_{i}(\alpha)$ is positive (respectively, negative), the extremal meridians are $\mu$ and the simple closed curve of slope  $\frac{1}{1}$ (respectively $-\frac{1}{1}$).   It follows that the train track $B^G(\alpha_1, \alpha_2, \dots, \alpha_{n})\cap\partial M$ fully carries the open interval of slopes that is bounded by these extremal meridians and contains all other meridians.

\section{Connected sums of fibered knots are persistently foliar}\label{proofofmaintheorem}

\begin{thm}\label{main} Suppose $\kappa_1$ and $\kappa_2$ are nontrivial fibered knots in $S^3$.
Any nontrivial slope on $\kappa=\kappa_1\# \kappa_2$ is strongly realized  by a co-oriented taut foliation  that  has a unique   minimal set, disjoint from   $\partial N(\kappa)$.  Hence $\kappa_1\#\kappa_2$  is persistently foliar.
\end{thm}


\begin{cor} \label{onesummandenough}  Suppose $\kappa=\kappa_1\#\cdots \#\kappa_n$ is a connected sum of  knots.  If at least one of the $\kappa_i$ is a nontorus alternating or Montesinos knot or a connected sum of fibered knots, then $\kappa$ is persistently foliar.\qed
\end{cor}

\begin{proof}
All Montesinos and prime alternating knots are either persistently foliar or fibered, by the results of \cite{DR2, DR3, DR4}.  Thus the result follows immediately from  Theorem~\ref{main}  and Corollary~\ref{consum}.
\end{proof}

\begin{cor} \label{alternatingmontesinos2}
Suppose $\kappa$ is a composite knot with a summand that is a nontorus alternating or Montesinos knot or the connected sum of two fibered knots,  and  $\widehat{X}_{\kappa}$ is a manifold obtained by non-trivial Dehn surgery along $\kappa$.   Then  $\widehat{X}_{\kappa}$ contains a co-oriented taut foliation;  hence, $\kappa$ satisfies the L-space Knot Conjecture.
\end{cor}

 We prove Theorem \ref{main} in the sections that follow.  First, in Section \ref{s: the spine}, we describe the spine, $\Sigma$, underlying the branched surface, $B$, that carries the minimal set of the desired foliations.  In Section \ref{s: firstsmoothing} we describe compatible co-orientations on $\Sigma$, smoothing it to obtain $B$.  In Section \ref{s: complement of B} we give a precise description of the complementary regions of $B$.  In Section \ref{s: no compact leaves} we prove that $B$ carries no compact leaves, and in Section \ref{s: fully carries} we prove that $B$ fully carries a lamination.  Finally, in Section \ref{s: extends to foliation}, we show that this lamination extends to a family of co-oriented taut foliations with unique common minimal set, carried by $B$, that strongly realize all boundary slopes.

\subsection{The spine $\Sigma$} \label{s: the spine}

Let $\kappa=\kappa_1\#\kappa_2\subset S^3$ be a connected sum, where   each of the knots $\kappa_1$ and $\kappa_2$  is  nontrivial and fibered, with fibers   $F_1$ and $F_2$  respectively.  Let $F$  denote the  band connect sum of  $F_1$ and $F_2$; so   $F$ is a fiber for $\kappa$  \cite{GabMurasugi}.

Let $P$ denote a summing sphere for this connected sum, cutting $F$ into $F_1$ and $F_2$.   Set $$A := P\cap (S^3\setminus \Int N(\kappa)) = P \cap X_\kappa.$$ Choose the isotopy representatives of  $\phi$ and $P$ so that  $A|_F = D(\beta)$ for an arc $\beta$ properly embedded in $F$.  Thus $\phi(\beta) = \beta$ and  the endpoints of $\beta$ are fixed points of $\phi$.

View $F$ as a compact surface properly embedded in $X_\kappa$, and view $F_1$ and $F_2$ as compact surfaces properly embedded in $X_\kappa|_A$.  Denote the component of $X_\kappa|_A$ containing $F_i$ by $X_{\kappa_i}$;  we observe that $X_{\kappa_i}$ is indeed homeomorphic to the complement of $\kappa_i$.   Let $T =\partial X_\kappa$. 

To simplify the exposition, we focus on the case that both $\kappa_1$ and $\kappa_2$ have right-veering monodromy. The case that they both have left-veering monodromy follows symmetrically. We address the remaining case, that $\kappa$ has monodromy that is neither right- nor left-veering, in Section~\ref{Remaining}.  

Choose non-separating, tight, properly embedded oriented arcs $\alpha_1$ in $F_1$ and $\alpha_2$ in $F_2$ disjoint from $\beta$ and such that $\phi(\alpha_i) \neq \alpha_i$, $i = 1, 2$.  Consider $D(\alpha_i) \in X_{\kappa_i}$.  Set  $\Sigma_0=T\cup A\cup F \cup D(\alpha_1) \cup D(\alpha_2).$       Notice that $\Sigma_0$ is not yet a spine as two surfaces meet transversely along $\beta$. To remedy this, isotope $F_2$ so that $F_2$ remains a properly embedded surface in $X_{\kappa_2}$, but $F_2\cap A$ is an isotopy representative of $\beta$ in $A$ that   meets $\beta$ transversely in a single point. (We could instead have chosen this representative to be disjoint from $\beta$.  Now set 
$$\Sigma=T\cup A\cup F_1\cup F_2 \cup D(\alpha_1) \cup D(\alpha_2).$$
To simplify notation, set $D_i=D(\alpha_i)$ for $i=1,2$.  
 
Finally, isotope $\Sigma$ into the interior of $X_\kappa$, so that $T$ is parallel to $\partial X_{\kappa}$, with the annulus $A$ still contained in the summing sphere $P$.

\subsection{The co-oriented branched surface $B$} \label{s: firstsmoothing} 

In this section, we describe a smoothing  of $\Sigma$ by fixing a compatible choice of co-orientations on the sectors of $\Sigma$. 

\begin{figure}[ht]
\labellist
\small
\pinlabel $T$ at 245 464
\pinlabel $T$ at 485 464
\pinlabel $F_1$ at 355 455
\pinlabel $-$ at 258 440
\pinlabel $-$ at 400 470
\pinlabel $+$ at 500 440
\pinlabel $A$ at 185 345
\pinlabel $+$ at 580 345
\pinlabel $-$ at 258 250
\pinlabel $+$ at 360 235
\pinlabel $-$ at 500 250
\pinlabel $T$ at 285 230
\pinlabel $T$ at 525 230
\pinlabel $F_2$ at 415 212
\endlabellist
\begin{center}
\includegraphics[scale=.35]{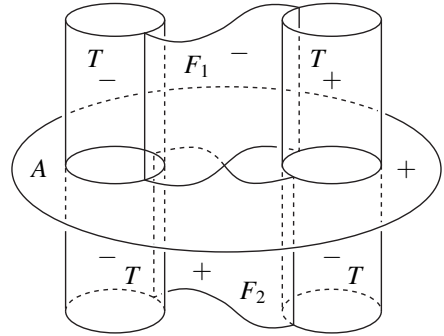}
\end{center}
\caption{The co-orientations on $\Sigma$ near $A$. }\label{SigmanearA}
\end{figure}

Choose a regular neighbourhood $N(A)$ of $A$ in $\Sigma$ such that the closure of $N(A)$   is disjoint from $D_1\cup D_2$, and fix co-orientations on the sectors of $N(A)$ as shown in Figure~\ref{SigmanearA}.   Give $F_1$ and $F_2$ the co-orientations that agree, respectively, with the co-orientations of $F_1\cap N(A)$ and  $F_2\cap N(A)$. Choose any co-orientations on $D_1$ and $D_2$. These induce orientations on $\alpha_1$ and $\alpha_2$.
Finally, cut $T$ open along $ \mu_0(\alpha_1)\cup\mu_0(\alpha_2)$ (as defined in Notation~\ref{notation: meridian})   and assign co-orientations to  the  resulting annuli components to agree with those of $T\cap N(A)$. We have thus described  co-orientations on the sectors of $\Sigma$.

\begin{figure}[ht]
\labellist
\small
\pinlabel $T$ at 160 540
\pinlabel $T$ at 325 540
\pinlabel $-$ at 355	 540
\pinlabel {Type C meridional smoothing} [Bl] at 385 530
\pinlabel $F_1$ at 225 475
\pinlabel $-$ at 265 465
\pinlabel $-$ at 160 450
\pinlabel $+$ at 355 450
\pinlabel $A$ at 95 430
\pinlabel $\mu_2$ [Bl] at 378 425
\pinlabel $+$ at 450 430
\pinlabel $\mu_1$ at 123 388
\pinlabel $-$ at 160 330
\pinlabel $+$ at 265 340
\pinlabel $+$ at 355 330
\pinlabel $F_2$ at 297 325
\pinlabel {Type C} [tr] at 130 305
\pinlabel {meridional smoothing} [tr] at 130 270
\pinlabel $+$ at 160 255
\pinlabel $T$ at 193 252
\pinlabel $T$ at 363 252
\pinlabel {Branching locus} [Bl] at 270 80
\pinlabel {on the annulus A} [Bl] at 270 45
\endlabellist
\begin{center}
\includegraphics[scale=.35]{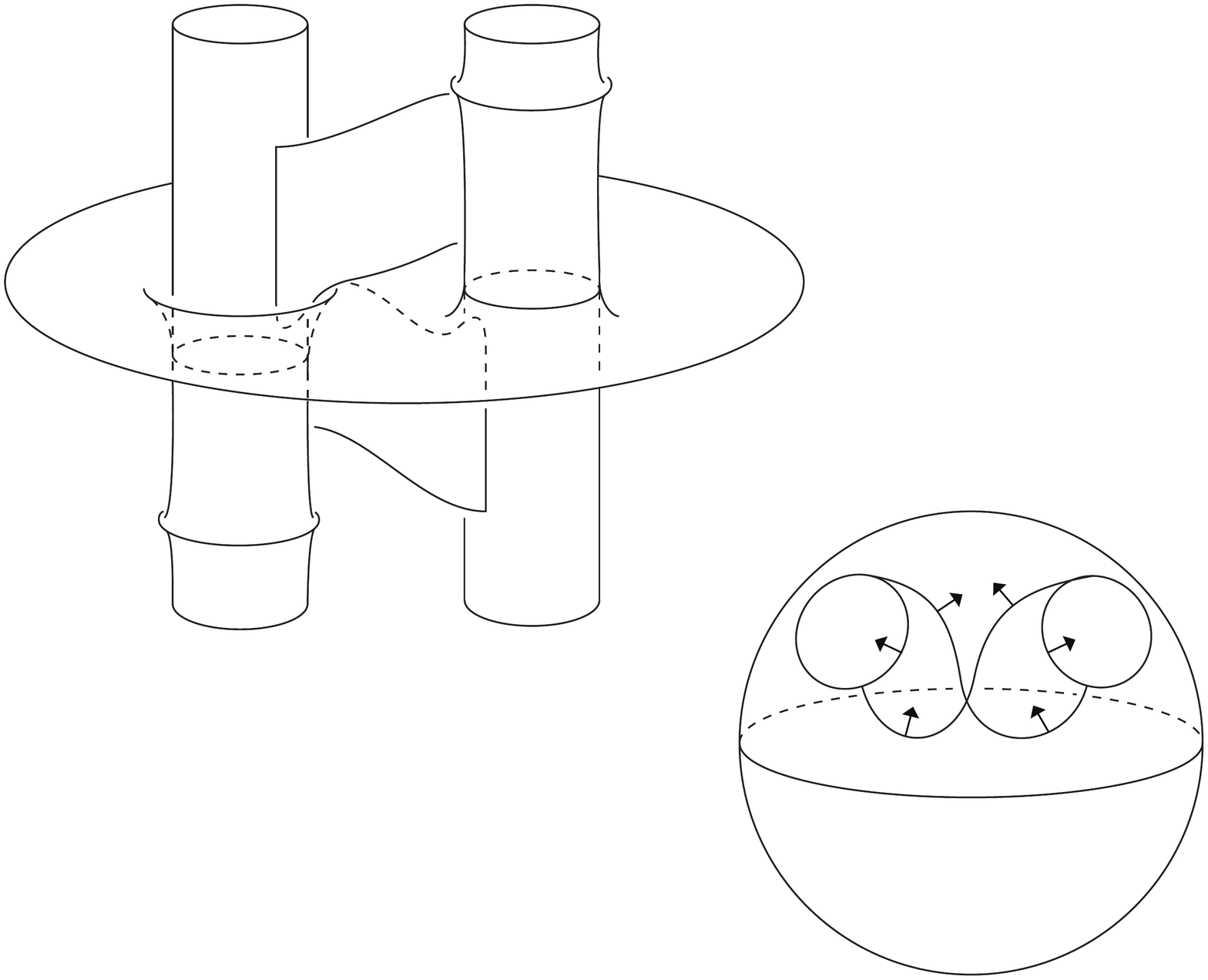}
\end{center}
\caption{The branched surface near $A$, also showing the co-orientations of $T$ near $ \mu_0(\alpha_1)$ and $\mu_0(\alpha_2)$. }\label{BnearA}
\end{figure}

It is straightforward to check that this choice of co-orientations on the sectors of $\Sigma$ determines a compatible smoothing of $\Sigma$ to a branched surface. Call this branched surface $B$. The smoothings restricted to $N(A)$ are shown in Figure~\ref{BnearA}. Those near $\mu_0(\alpha_1)$ and $\mu_0(\alpha_2)$ are shown in Figures~\ref{cuspintropos} and \ref{cuspintro}, and called \emph{Type C}, and those near  $\mu_1(\alpha_1)$ and $\mu_1(\alpha_2)$ are shown in 
Figure~\ref{posB}, and called \emph{Type B}.  
We note that the choice of co-orientations on the sectors of $\Sigma$ is motivated by the theory developed in \cite{DR2,DR3}, as is the terminology Type C (for cusp) and Type B.

\begin{figure}[ht]
\begin{center}
\labellist
\small
\pinlabel {\footnotesize $F$} at 50 550
\pinlabel {\footnotesize $T$} at 85 489
\pinlabel {\footnotesize $D$} at 135 518
\pinlabel {\footnotesize Triple point at $x$} at 535 430
\pinlabel {\footnotesize $F$} at 305 550
\pinlabel {\footnotesize $F$} at 292 506
\pinlabel {\footnotesize $F$} at 212 485
\pinlabel {\scriptsize $D$} at 225 526
\pinlabel {\footnotesize $T$} at 247 492
\pinlabel $\simeq$ at 180 515
\pinlabel $D$ at 557 535
\pinlabel $F$ at 490 505
\pinlabel {\scriptsize $F$} at 579 484
\pinlabel $T$ at 653 515
\pinlabel $T$ at 641 486
\pinlabel {\scriptsize $T$} at 509 465
\pinlabel {\scriptsize $\mp$} at 190 402
\pinlabel {\scriptsize $\pm$} at 145 357
\pinlabel {\scriptsize $\mp$} at 120 280
\pinlabel {\scriptsize $\pm$} at 75 260
\pinlabel {\footnotesize $x$} [tr] at 90 328
\pinlabel {\footnotesize $x$} [tr] at 88 296
\pinlabel {\footnotesize $y$} [bl] at 182 345
\pinlabel {\footnotesize $y$} [tl] at 172 328
\pinlabel {\scriptsize $\mp$} at 410 380
\pinlabel {\scriptsize $\pm$} at 375 365
\pinlabel {\scriptsize $\pm$} at 350 300
\pinlabel {\scriptsize $\mp$} at 380 280
\pinlabel {\footnotesize Branch} at 203 190
\pinlabel {\footnotesize curves} at 203 165
\pinlabel {\scriptsize $\mp$} at 340 170
\pinlabel {\scriptsize $\pm$} at 300 152
\pinlabel {\scriptsize $\pm$} at 265 81
\pinlabel {\scriptsize $\mp$} at 303 70
\pinlabel $F$ at 655 215
\pinlabel $T$ at 545 184
\pinlabel $D$ at 627 183
\pinlabel $F$ at 690 165
\pinlabel $T$ at 550 135
\pinlabel {\footnotesize Triple point at $y$} at 625 95
\endlabellist
\includegraphics[scale=.4]{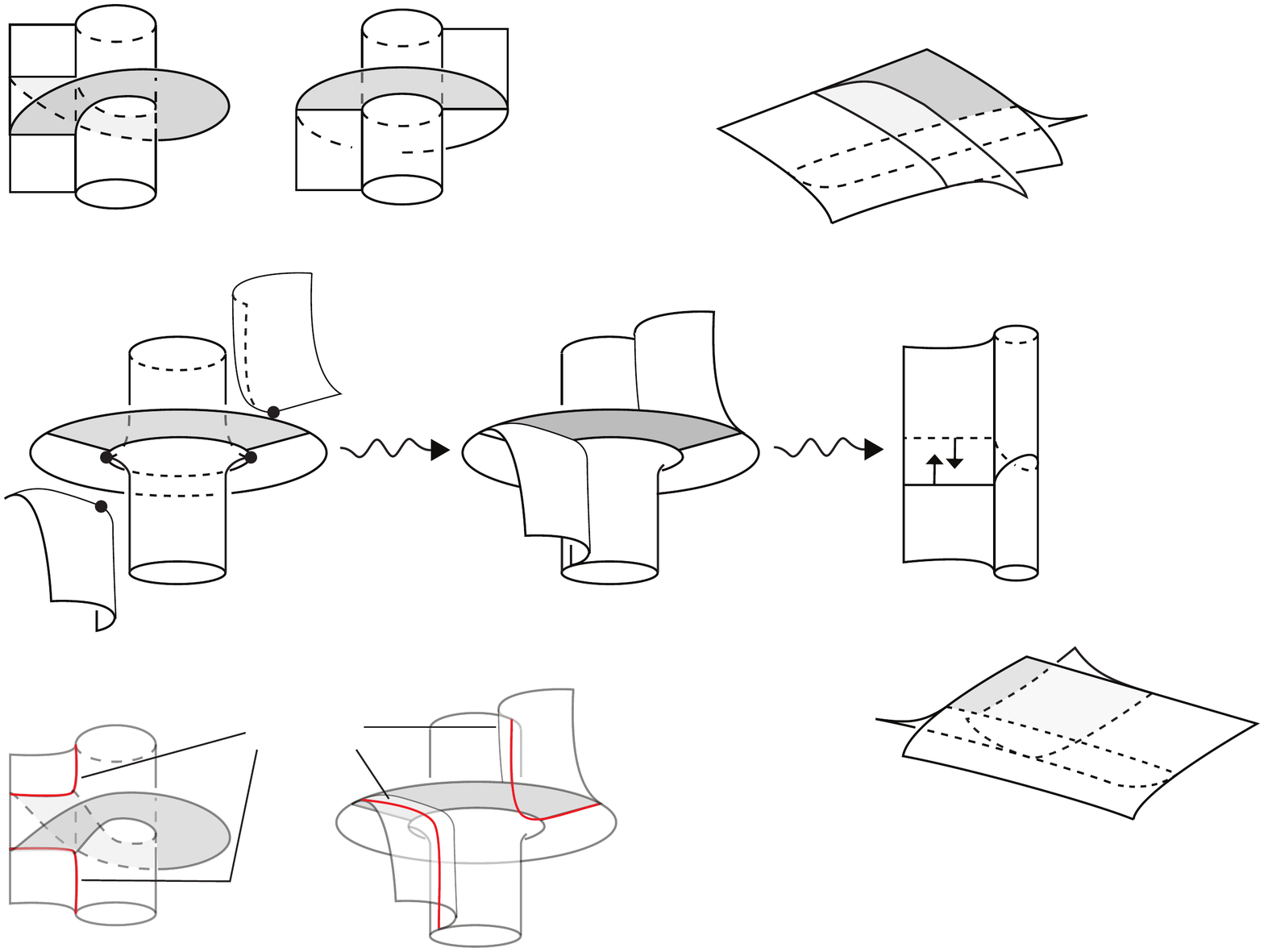}
\end{center}
\caption{$B$ in a neighbourhood of a positive Type C crossing: introducing a meridian cusp in the region containing $\kappa$.}
\label{cuspintropos}
\end{figure}

\begin{figure}[ht]
\labellist
\small
\pinlabel {\footnotesize $F$} at 50 550
\pinlabel {\footnotesize $T$} at 85 489
\pinlabel {\footnotesize $D$} at 137 522
\pinlabel {\footnotesize Triple point at $x$} at 540 430
\pinlabel $\simeq$ at 180 520
\pinlabel {\footnotesize $F$} at 210 550
\pinlabel {\scriptsize $F$} at 290 530
\pinlabel {\footnotesize $F$} at 300 488
\pinlabel {\scriptsize $D$} at 227 511
\pinlabel {\footnotesize $T$} at 247 545
\pinlabel $D$ at 557 470
\pinlabel $F$ at 480 500
\pinlabel {\scriptsize $F$} at 580 516
\pinlabel $T$ at 641 520
\pinlabel $T$ at 655 490
\pinlabel {\scriptsize $T$} at 502 539
\pinlabel {\scriptsize $\mp$} at 80 390
\pinlabel {\scriptsize $\pm$} at 120 370
\pinlabel {\scriptsize $\mp$} at 143 291
\pinlabel {\scriptsize $\pm$} at 190 250
\pinlabel {\footnotesize $x$} [Br] at 93 330
\pinlabel {\footnotesize $x$} [br] at 88 355
\pinlabel {\footnotesize $y$} [tl] at 182 302
\pinlabel {\footnotesize $y$} [Bl] at 172 330
\pinlabel {\scriptsize $\pm$} at 380 370
\pinlabel {\scriptsize $\mp$} at 350 355
\pinlabel {\scriptsize $\mp$} at 375 285
\pinlabel {\scriptsize $\pm$} at 410 270
\pinlabel {\scriptsize $\pm$} at 303 160
\pinlabel {\scriptsize $\mp$} at 265 150
\pinlabel {\scriptsize $\mp$} at 300 81
\pinlabel {\scriptsize $\pm$} at 340 70
\pinlabel {\footnotesize Branch} at 201 65
\pinlabel {\footnotesize curves} at 201 40
\pinlabel $F$ at 650 90
\pinlabel $T$ at 533 157
\pinlabel $D$ at 615 120
\pinlabel $F$ at 688 140
\pinlabel $T$ at 517 132
\pinlabel {\footnotesize Triple point at $y$} at 630 65
\endlabellist
\begin{center}
\includegraphics[scale=.4]{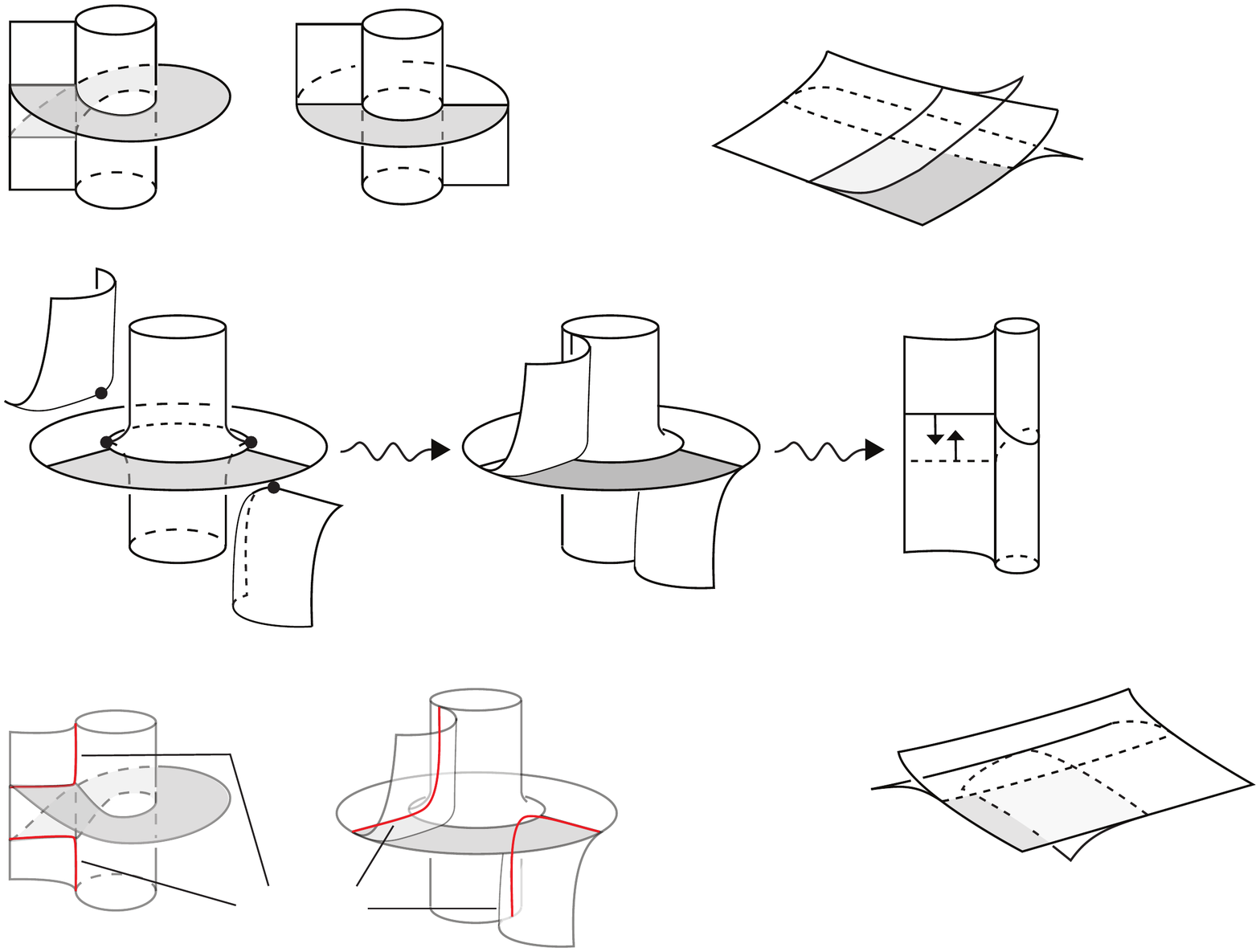}
\end{center}
\caption{$B$ in a neighbourhood of a negative Type C crossing: introducing a meridian cusp in the region containing $\kappa$.}
\label{cuspintro}
\end{figure}

\begin{figure}[ht]
\labellist
\small
\pinlabel {\tiny $D$} at 92 539
\pinlabel {\tiny $F$} at 145 518
\pinlabel {\tiny $F$} at 70 490
\pinlabel {\tiny $\mp$} at 114 567
\pinlabel {\tiny $\mp$} at 135 550
\pinlabel {\tiny $\pm$} at 162 528
\pinlabel {\tiny $\mp$} at 114 495
\pinlabel {\tiny $\mp$} at 85 493
\pinlabel {\tiny $F$} at 398 553
\pinlabel {\tiny $D$} at 310 527
\pinlabel {\tiny $F$} at 373 518
\pinlabel {\tiny $F$} at 305 490
\pinlabel {\tiny $\mp$} at 349 567
\pinlabel {\tiny $\mp$} at 326 530
\pinlabel {\tiny $\pm$} at 375 555
\pinlabel {\tiny $\mp$} at 349 495
\pinlabel {\tiny $\mp$} at 323 495
\pinlabel {$=$} at 430 518
\pinlabel {\tiny $\mp$} at 510 560
\pinlabel {\tiny $\mp$} at 486 530
\pinlabel {\tiny $\pm$} at 545 555
\pinlabel {\tiny $\mp$} at 510 505
\pinlabel {\tiny $\mp$} at 475 500
\pinlabel $F$ at 478 450
\pinlabel $D$ at 485 390
\pinlabel $T$ at 470 315
\pinlabel $x$ [l] at 258 396
\pinlabel $x$ [Bl] at 537 380
\pinlabel {Triple point} [Bl] at 590 408
\pinlabel {at $x$} [Bl] at 590 378
\pinlabel {\tiny $F$} at 164 277
\pinlabel {\tiny $D$} at 53 265
\pinlabel {\tiny $F$} at 135 240
\pinlabel {\tiny $F$} at 65 215
\pinlabel {\tiny $\pm$} at 70 270
\pinlabel {\tiny $\mp$} at 112 250
\pinlabel {\tiny $\mp$} at 150 245
\pinlabel {\tiny $\pm$} at 80 218
\pinlabel {\tiny $F$} at 400 275
\pinlabel {\tiny $D$} at 310 250
\pinlabel {\tiny $F$} at 372 237
\pinlabel {\tiny $F$} at 302 215
\pinlabel {\tiny $\mp$} at 350 278
\pinlabel {\tiny $\pm$} at 325 255
\pinlabel {\tiny $\mp$} at 387 275
\pinlabel {\tiny $\mp$} at 350 220
\pinlabel {\tiny $\pm$} at 317 218
\pinlabel {$=$} at 431 245
\pinlabel {\tiny $\mp$} at 505 282
\pinlabel {\tiny $\pm$} at 477 250
\pinlabel {\tiny $\mp$} at 553 275
\pinlabel {\tiny $\mp$} at 505 220
\pinlabel {\tiny $\pm$} at 472 218
\pinlabel $T$ at 485 173
\pinlabel $D$ at 485 90
\pinlabel $F$ at 460 32
\pinlabel $x$ [l] at 258 88
\pinlabel $x$ [Bl] at 537 90
\pinlabel {Triple point} [Bl] at 590 115
\pinlabel {at $x$} [Bl] at 590 85
\endlabellist
\begin{center}
\includegraphics[scale=.4]{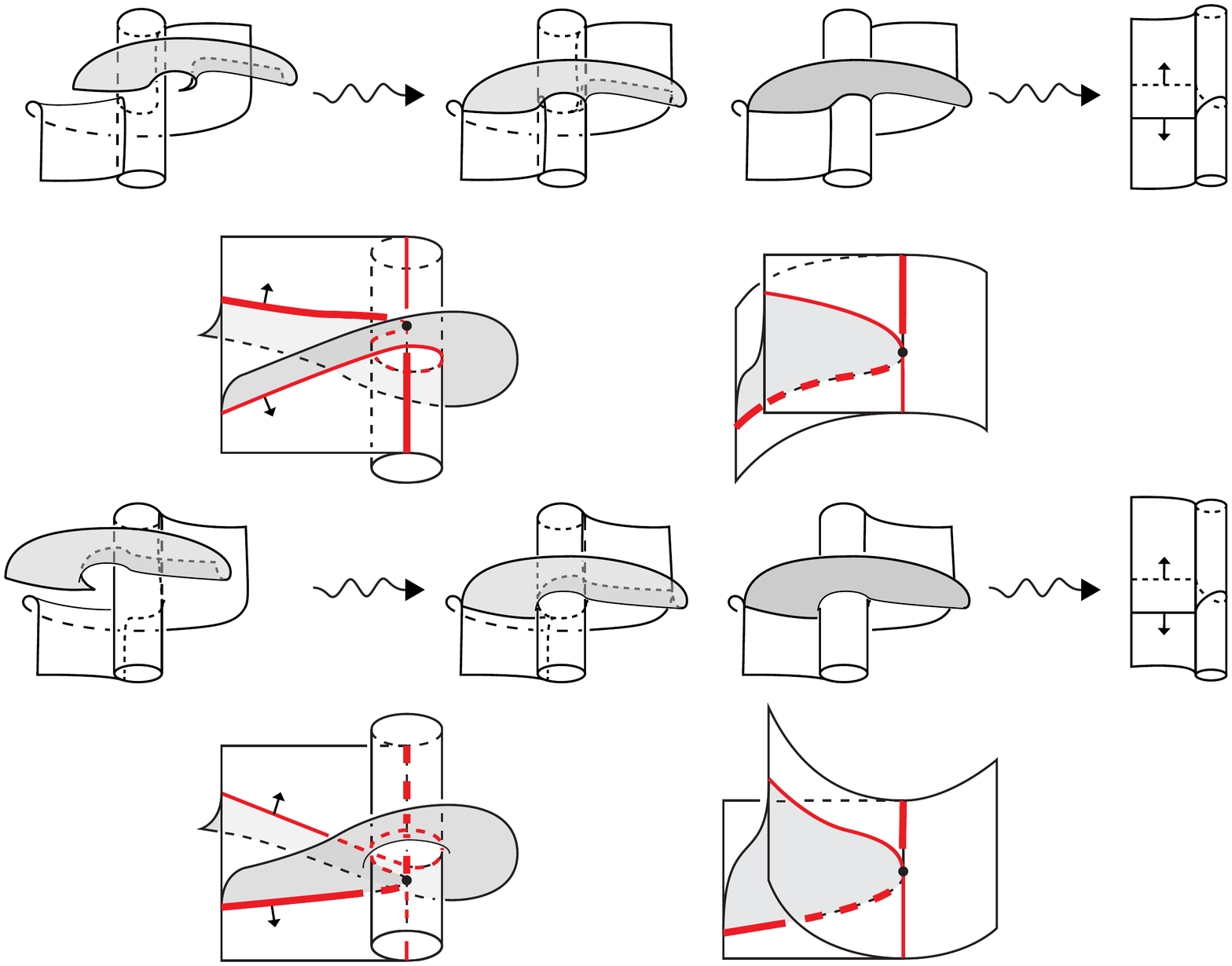}
\end{center}
\caption{$B$ in a neighbourhood of a positive Type B transition.}\label{posB}
\end{figure}

It is helpful for calculations to make note of the components of $\partial_vN(B)$ that result locally from each type of smoothing near a transition arc.  These are shown in  (red) boldface in Figure~\ref{sutures}.   

\begin{figure}[ht]
\labellist
\small
\pinlabel {Type A or C} [B] at 220 165
\pinlabel {Type B} [B] at 560 165
\endlabellist
\begin{center}
\includegraphics[scale=.3]{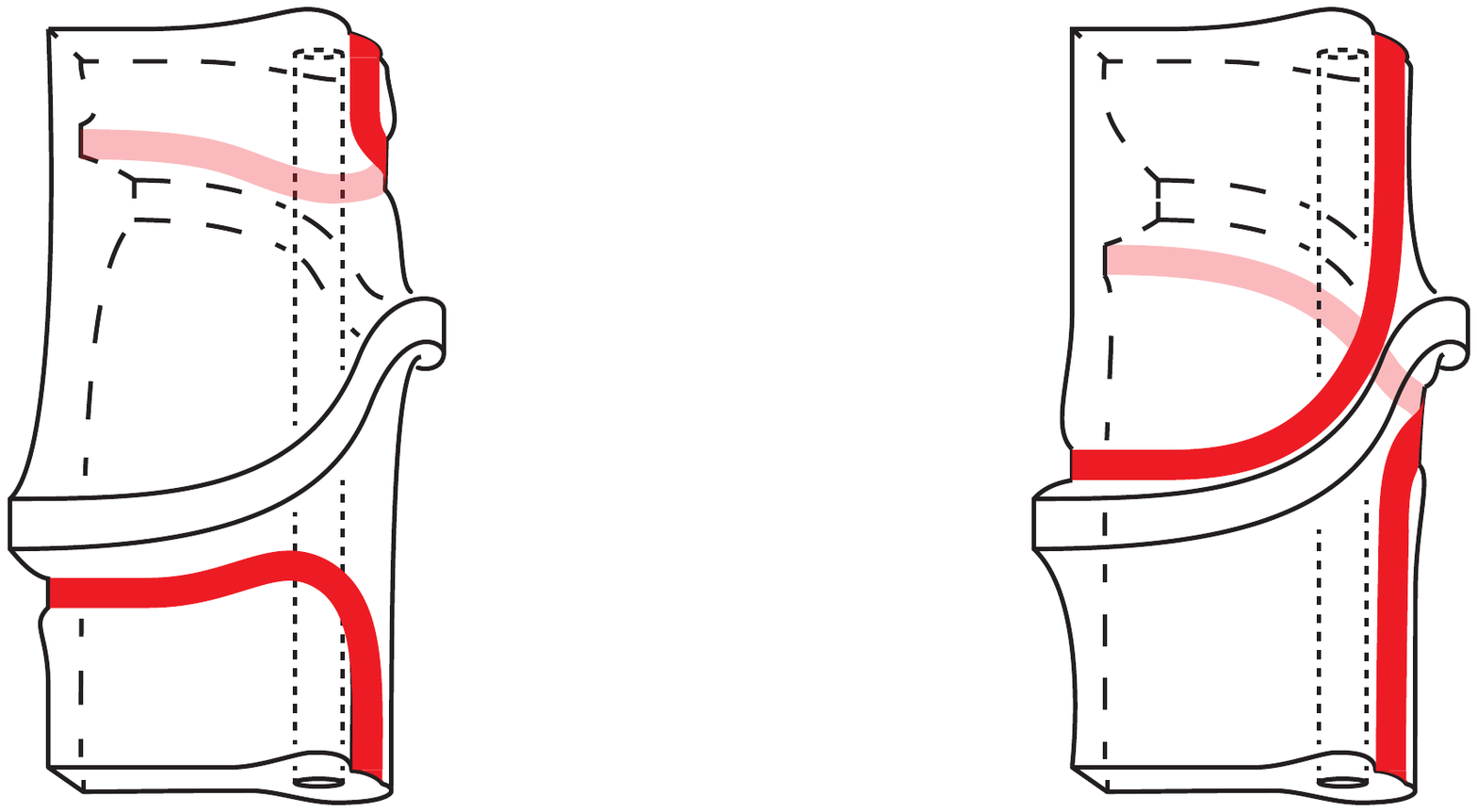}
\end{center}
\caption{$N(B)$, showing the sutures, $\partial_vN(B)$, at smoothings of Type A, B, and C.}  \label{sutures}
\end{figure}
 
\subsection{The three complementary regions of $B$} \label{s: complement of B}

For each $i$, let $g_i$ denote the genus of  $F_i$.      We now describe the components of the sutured manifold $\left(X_\kappa \setminus \Int N(B), \partial_v(N(B)\right)$, commonly referred to as the \emph{complementary regions of $B$}.   Clearly there are three, one of which contains $\partial X_\kappa$, and one lying in each $X_{\kappa_i}$.  Let $V_{1}$ and $V_{2}$ be the annuli of vertical boundary with cores $ \mu_0(\alpha_1)$ and $\mu_0(\alpha_2)$.  

\begin{prop} \label{Gabaicomplements} The sutured manifold  $(X_\kappa \setminus \mbox{int }N(B),\partial_v N(B))$  consists of the following three (sutured manifold) components: 
\begin{enumerate}
\item  $\left(\partial X_\kappa \times I, V_1 \cup V_2 \right)$, 
\item $(F_1'\times I,\partial F_1'\times I)$, and 
\item $(F_2'\times I,\partial F_2'\times I)$,
\end{enumerate}
where $F_i'=F_i|_{\alpha_i}$ is a surface with two boundary components and genus $g_i-1$. 
\end{prop}

\begin{proof}   $T$ is parallel to $\partial X_\kappa$.    Moreover, each of the two Type C neighbourhoods of $B$ introduces a single meridian suture. Hence, the complementary region containing  $\partial X_\kappa$  is isomorphic to the sutured manifold described in (1).  

Set $B_1^G=\langle F_1,D_1\rangle$ and $B_2^G=\langle F_2, D_2\rangle$.  By Lemma~\ref{BGabai}, it suffices to show that the remaining  components complementary to $\Int N(B)$ are isomorphic as sutured manifolds to  the  closed complements of $B_1^G$ and $B_1^G$, respectively.   Let   $Y_i$ be the  component   that lies   in $X_{\kappa_i}$. Forgetting the sutured manifold  structure of $\partial Y_i$, the compact 3-manifold $Y_i$ is a genus $2g_i -  2$ handlebody, and hence is homeomorphic to $F_i'\times I$. It suffices, therefore, to prove that this homeomorphism can be chosen so that $\partial'_v Y_i$ is mapped to   $ \partial F_i\times I$.  
Away from   $\partial A$   and the  crossings  $D_i\cap T$, the core of $\partial'_v Y_i$ runs along $T$, parallel to $\partial F_i$. At the crossings, this core combines with the arcs $D_i\cap F_i$ (topologically) as it does in $B^G$; see  Figure~\ref{sutures}.

At   $\partial A$, this core wraps partway about $\partial A$, but disjointly from $\partial F_i$. Hence $(Y_i,\partial'_v Y_i)$ is isomorphic to $(F_i'\times I,\partial F_i'\times I).$ This is illustrated in Figure~\ref{sutureswork}.
\end{proof}

\begin{figure}[ht]
\labellist
\small
\pinlabel {Type C} at 145 570
\pinlabel {Type B} at 245 570
\pinlabel {Type C} at 515 570
\pinlabel {Type B} at 600 570
\pinlabel {\footnotesize $\phi(\alpha)$} [Br] at 120 405
\pinlabel {\footnotesize $\alpha$} [Br] at 183 405
\pinlabel {\footnotesize $\phi(\alpha)$} [Bl] at 201 405
\pinlabel {\footnotesize $\alpha$} [Bl] at 268 405
\pinlabel {\footnotesize $\phi(\alpha)$} [Br] at 483 405
\pinlabel {\footnotesize $\phi(\alpha)$} [Bl] at 520 405
\pinlabel {\footnotesize $\alpha$} [Br] at 610 405
\pinlabel {\footnotesize $\alpha$} [Bl] at 645 405
\endlabellist
\begin{center}
\includegraphics[scale=.4]{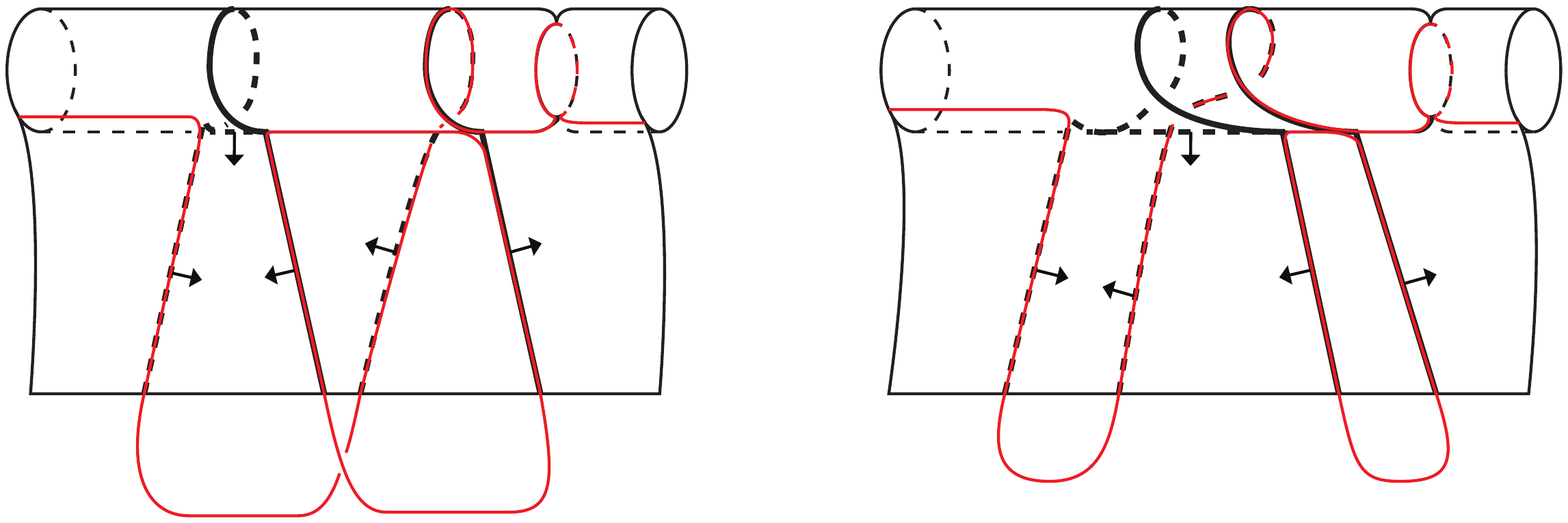}
\end{center}
\caption{The sutures of $B$  agree with those of $B^G$.}\label{sutureswork}
\end{figure}

\begin{cor} \label{surgerycomplements}Let $\widehat{M}$ denote a closed 3-manifold obtained by Dehn filling of slope $\frac{p}{q}$ along $\kappa$. 
The complement $(\widehat{M}\setminus \mbox{int }N(B),\partial_v N(B))$ consists of the following three components: 
\begin{enumerate}
\item  A solid torus whose meridian has minimal geometric intersection number $2|q|$ with   $\mu_0(\alpha_1) \cup \mu_0(\alpha_2)$,  
\item $(F_1'\times I,\partial F_1'\times I)$, and 
\item $(F_2'\times I,\partial F_2'\times I)$,
\end{enumerate}
where $F_i'$ is a surface with two boundary components and genus $g_i-1$.
\end{cor}

\begin{proof}
Consider the complementary  component  that contains $\partial X_\kappa$. After Dehn filling $\partial X_\kappa$ with slope $p/q$, this  component transforms to a solid torus with meridian intersecting each of the curves $\mu_0(\alpha_1)$ and $\mu_0(\alpha_2)$ minimally in $|\langle  1/0,p/q\rangle|  = |q|$ points.
 \end{proof}

\subsection{Any leaf carried by $B$ is noncompact.} \label{s: no compact leaves}

\begin{prop} \label{prop:leavesnoncpt} Any surface carried by $B$ has nonempty intersection with every branch of $B$, and is noncompact. In particular, $B$ does not carry a torus.
\end{prop}

\begin{proof}   
Let $L$ be any nonempty surface carried by $B$, and let $B_L$ be the sub-branched surface of $B$ that fully carries $L$. If $B$ fully carries $L$, then $B_L=B$. In general, $B_L$ is a union of   sectors of $B$. 

We first observe that $B_L$ must contain a sector of $F$ that has nonempty intersection with $T$.  Suppose by way of contradiction that it does not.  Since the sink directions on $D_i \cap F$ point into $F$ for each $i$,  $B_L$ contains such a sector of $F$ whenever $B_L$ contains $D_i$;  hence we may assume that $B_L$ contains no $D_i$. But if $B_L$ does not contain $D_1$ or $D_2$, it can contain a sector of $F$ only when it contains every sector of $F$. Hence, we may assume that $B_L$ does not contain $D_1$, $D_2$, or any sector of $F$, and therefore does not contain any sector of $T$, since a cusp direction points from $T$ into $F$ at Type C smoothings. But it then follows that $B_L$ cannot contain a sector of $A$, and hence is empty, an impossibility.

Thus, $B_L$ contains a sector $F_0$ of $F$ that has nonempty intersection with $T$.  Since $F_0$ has an arc of boundary along $T$ with outward-pointing cusp direction,   $L \cap N(T)$ contains a proper embedding of a ray $[0,\infty)$ carried by a meridian of $T$;  hence, $L$ is not compact. 
\end{proof}

\subsection{$B$ fully carries a lamination $\mathcal L$} \label{s: fully carries} 

The branched surface $B$ might contain sink or half sink disks. However, using ideas from \cite{Rfib2}, it is straightforward to show that it can be split to a  branched surface that contains no sink or half sink disk.

\begin{notation}
Given a sequence of oriented arcs $\alpha_{i,1}, \alpha_{i,2}, \alpha_{i,3}, \ldots, \alpha_{i,n}$ embedded in $F_i$, let $D_{i,j} = \alpha_{i, j}\times \left[\frac{j-1}{n},\frac{j}{n}\right]$, oriented so that $\partial D_{i,j}$ contains $\alpha_{i,j} \times \left\{\frac{j-1}{n}\right\}$ as a positively oriented subarc.   Let $F_{i,j} = F_i \times \left\{\frac{j}{n}\right\}$.
\end{notation}

\begin{prop} \label{canbemadelaminar}
The branched surface $B$ can be split open  to a laminar branched surface $B'$ homeomorphic to the spine $$T \cup A \cup \left(\bigcup_{j=1}^{n_1} F_{1, j-1} \cup D_{1,j}  \right) \cup \left(\bigcup_{j=1}^{n_2} F_{2, j-1} \cup D_{2,j}  \right),$$
where $\phi(\alpha_i) = \alpha_{i,0}, \alpha_{i,1}, \alpha_{i,2}, \ldots, \alpha_{i,{n_i}} = \alpha_i$ is a $0$-end-effective sequence in $F_i$.

The complement of $B'$ has $2n+1$ components:
\begin{enumerate}
\item   $\left(\partial X_\kappa \times I, V'_1 \cup V'_2 \right)$, where $V'_1$ and $V'_2$ are disjoint meridional annuli,   
\item $n_1$ copies of $(F_1'\times I,\partial F_1'\times I)$, and 
\item $n_2$ copies of $(F_2'\times I,\partial F_2'\times I)$,
\end{enumerate} 
where each $F_i'$ is a surface with two boundary components and genus $g_i-1$.  
\end{prop}

\begin{proof} Recall our assumption that all transition arcs are positive; hence Type B smoothings occur at $\delta_1$  and Type C smoothings occur at $\delta_0$ for each arc $\alpha_i$, $i=1,2$.

\begin{figure}[ht]
\labellist
\small
\pinlabel $F_{i,0}$ at 125 365
\pinlabel $F_{i,1}$ at 300 365
\pinlabel $D_{i,2}$ [tr] at 105 330
\pinlabel $D_{i,1}$ [tl] at 335 270
\pinlabel $+$ at 135 290
\pinlabel $-$ at 280 290
\pinlabel $\mp$ at 620 418
\pinlabel $\pm$ at 465 380
\pinlabel $\pm$ at 550 374
\pinlabel $\mp$ at 550 280
\pinlabel $\pm$ at 458 190
\pinlabel $\mp$ at 612 152
\pinlabel $D_{i,2}$ [l] at 685 243
\pinlabel $D_{i,1}$ [l] at 702 323
\endlabellist
\begin{center}
\includegraphics[scale=.4]{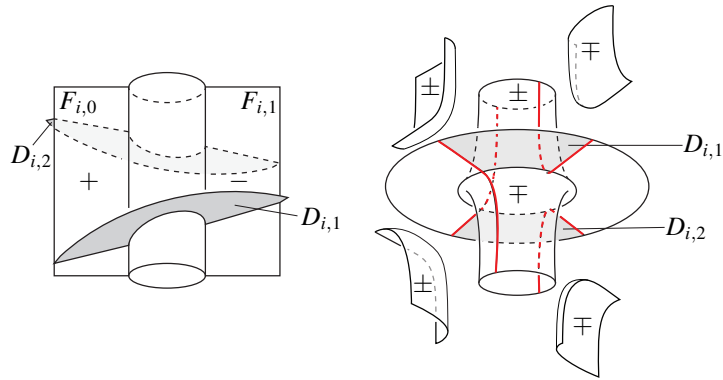}
\end{center}
\caption{Introducing a meridian cusp after splitting. Example with two copies of fiber $F_i$.}
\label{cuspintro2}
\end{figure}

Applying Theorem~\ref{Gcanbemadelaminar} to $B^G(\alpha_i)$, for each $i=1,2$,  there are $0$-end-effective  sequences $\phi(\alpha_i) = \alpha_{i,0}, \alpha_{i,1}, \alpha_{i,2}, \ldots, \alpha_{i,{n_i}} = \alpha_i$  such that  each branched surface $B^G(\alpha_{i,1}, \alpha_{i,2}, \ldots, \alpha_{i,{n_i}})$ is laminar, and each train track $\tau_i = B^G(\alpha_{i,1}, \alpha_{i,2}, \ldots, \alpha_{i,{n_i}})\cap \partial M$, contains the meridian as a subtrack   containing $\delta_0(\alpha_i)$. Denote each meridian subtrack by $\mu_{\tau_i}$.  Recall that each $F_{i,j}$ is oriented consistently with $F_i$, $i=1,2$.

Setting $$\Sigma'=T \cup A \cup \left(\bigcup_{j=1}^{n_1} F_{1, j-1} \cup D_{1,j}  \right) \cup \left(\bigcup_{j=1}^{n_2} F_{2, j-1} \cup D_{2,j}  \right),$$ we describe a smoothing  of $\Sigma'$ by fixing a compatible choice of co-orientations on the sectors of $\Sigma'$. Indeed, co-orientations have been fixed for all sectors except those lying in $T$. We define co-orientations in the sectors of $T$ by choosing co-orientations on the two annuli obtained by cutting $T$ open along 
$\mu_{\tau_1}\cup  \mu_{\tau_2}$, choosing these co-orientations to agree with the co-orientations chosen on $T\cap N(A)$. 

It is straightforward to check that this choice of co-orientations on the sectors of $\Sigma'$ determines a compatible smoothing of $\Sigma'$ to a branched surface. Call this branched surface $B'$.
Under this smoothing, the two meridians  $\mu_{\tau_1}\cup  \mu_{\tau_2}$ become meridian cusps in the complementary region that contains $\partial X_{\kappa}$. This is illustrated in Figure~\ref{cuspintro2}.    Let $V'_i$ be the annulus of vertical boundary with core $\mu_{\tau_i}$. 

Since $B$ does not carry a torus, and $B'$ is a splitting of $B$,  $B'$ does not carry a torus. Moreover, since the sequences $\alpha_{i,0}, \alpha_{i,1}, \alpha_{i,2}, \ldots, \alpha_{i,{n_i}}$ are $\alpha$-sparse, neither $B^G(\alpha_{i,1}, \alpha_{i,2}, \ldots, \alpha_{i,{n_i}})$ has a sink disk or half sink disk;  thus, $B'$ has no sink disk or half sink disk.  So $B'$ is laminar.
\end{proof}

\begin{cor}\label{thelaminationL}
$B$ fully carries a lamination.
\end{cor}

\begin{proof}
The branched surface $B'$ described in Proposition~\ref{canbemadelaminar} is laminar, and hence  fully carries a lamination $\mathcal L$ \cite{Li0}. Since $B'$ is obtained by splitting $B$, $\mathcal L$ is also  fully carried by $B$.
\end{proof}

\subsection{$\mathcal L$ extends to co-oriented taut foliations that strongly realize all boundary slopes} \label{s: extends to foliation}

 \begin{prop} \label{strongdetection}
 For each slope $\gamma$ (not necessarily rational), the lamination $\mathcal L$ extends to a   co-oriented taut foliation  $\mathcal F_{\gamma}$ that strongly realizes $\gamma$. Each $\mathcal F_{\gamma}$ has a unique minimal set, fully carried by $B$.  
 \end{prop}

\begin{proof}  
The complementary region of $B$ that is not a product (as a sutured manifold) is the one containing   $\partial X_\kappa$:  $\left(\partial X_\kappa \times I, V'_1 \cup V'_2 \right)$, where $V'_1$ and $V'_2$ are disjoint meridional annuli. Denote this region by $Y_\partial$. We will now show that for  every nontrivial slope $\gamma$ (not necessarily rational), this region can be filled in by noncompact leaves that meet $\partial X_{\kappa}$ in parallel leaves of slope $\gamma$.

When $\gamma$ is rational, this region contains a properly embedded annulus $A_{\gamma}= \gamma\times I$. When $\gamma$ is not the meridian, any choice of co-orientation of $A_{\gamma}$  describes a smoothing of $\Sigma\cup A_{\gamma}$ to a  branched surface in $X_\kappa$ whose complementary regions are all products (as sutured manifolds). 

In general (when $\gamma$ is either rational, but not the longitude $\lambda$, or irrational), proceed instead as follows. Consider the essential annulus   $A_{\lambda}=\lambda\times I$.  Again, any choice of co-orientation of $A_{\lambda}$  describes a smoothing of $\Sigma_{\lambda}=\Sigma\cup A_{\lambda}$ to a  branched surface in $X_\kappa$ whose complementary regions are all products.  In particular, the complementary region $Y_\partial |_{A_{\lambda}}$ is a solid torus with two longitudinal sutures (one of which is $\partial M \setminus \lambda$).  Let $D_{\mu}$ be the product disk for this region, isotoped so that the essential arcs $D_{\mu}\cap A_{\lambda}$ are disjoint. The two distinct choices of orientation on $D_{\mu}$ give rise to two smoothings of $\Sigma_{\lambda}\cup D_{\mu}$; call the resulting branched surfaces $B_1$ and $B_2$. The isotopy representative of $D_{\mu}$ can be chosen so that the train tracks  $B_1\cap \partial X_{\kappa}$ and $B_2\cap \partial X_{\kappa}$ together fully carry all nontrivial, nonlongitudinal boundary slopes. The associated measures on these train tracks describe measured laminations that are fully carried by the sub-branched surfaces (not properly embedded) with spine $A_{\lambda}\cup D_{\mu}$. See Figure \ref{figure: realizes all slopes}.  (Alternatively, the branched surfaces $B_1$ and $B_2$ are laminar, and hence there exist co-oriented laminations fully carried by $B_1$ or $B_2$ that strongly realize any nontrivial, nonlongitudinal slope $\gamma$ \cite{Li}. The proof of the main result of \cite{Li} reveals that these foliations can be chosen to include $\mathcal L$ as a sublamination.) This argument can of course be repeated replacing $\lambda$ with any nontrivial rational slope.

\begin{figure}[ht]
\labellist
\small
\pinlabel {Branch curves} at 190 595
\pinlabel $\partial X_\kappa$ [Bl] at 25 545
\pinlabel ! at 88 548
\pinlabel ! at 285 548
\pinlabel {$Y_\partial$ (Identify annuli marked ``!'')} at 190 450
\pinlabel ! at 398 548
\pinlabel ! at 596 548
\pinlabel {Branch curve} at 460 448
\pinlabel {Half-disk $\times I$} [Bl] at 572 432
\pinlabel {product sutured} [Bl] at 572 410
\pinlabel {manifold component} [Bl] at 572 388
\pinlabel $D_\mu$ at 545 398
\pinlabel {Branch curve} at 190 255
\pinlabel $A_\lambda$ at 302 260
\pinlabel {(Cut open on $D_\mu$ and $A_\lambda$)} at 535 190
\pinlabel $\partial M$ at 170 95
\pinlabel $1$ at 230 100
\pinlabel $\mu$ at 283 100
\pinlabel {\scriptsize $1+x$} at 350 75
\pinlabel $\lambda$ at 413 125
\pinlabel $x$ at 396 75
\pinlabel {Slope = $x$ (in terms of $(\lambda, \mu)$)} at 320 22
\endlabellist
\begin{center}
\includegraphics[scale=.5]{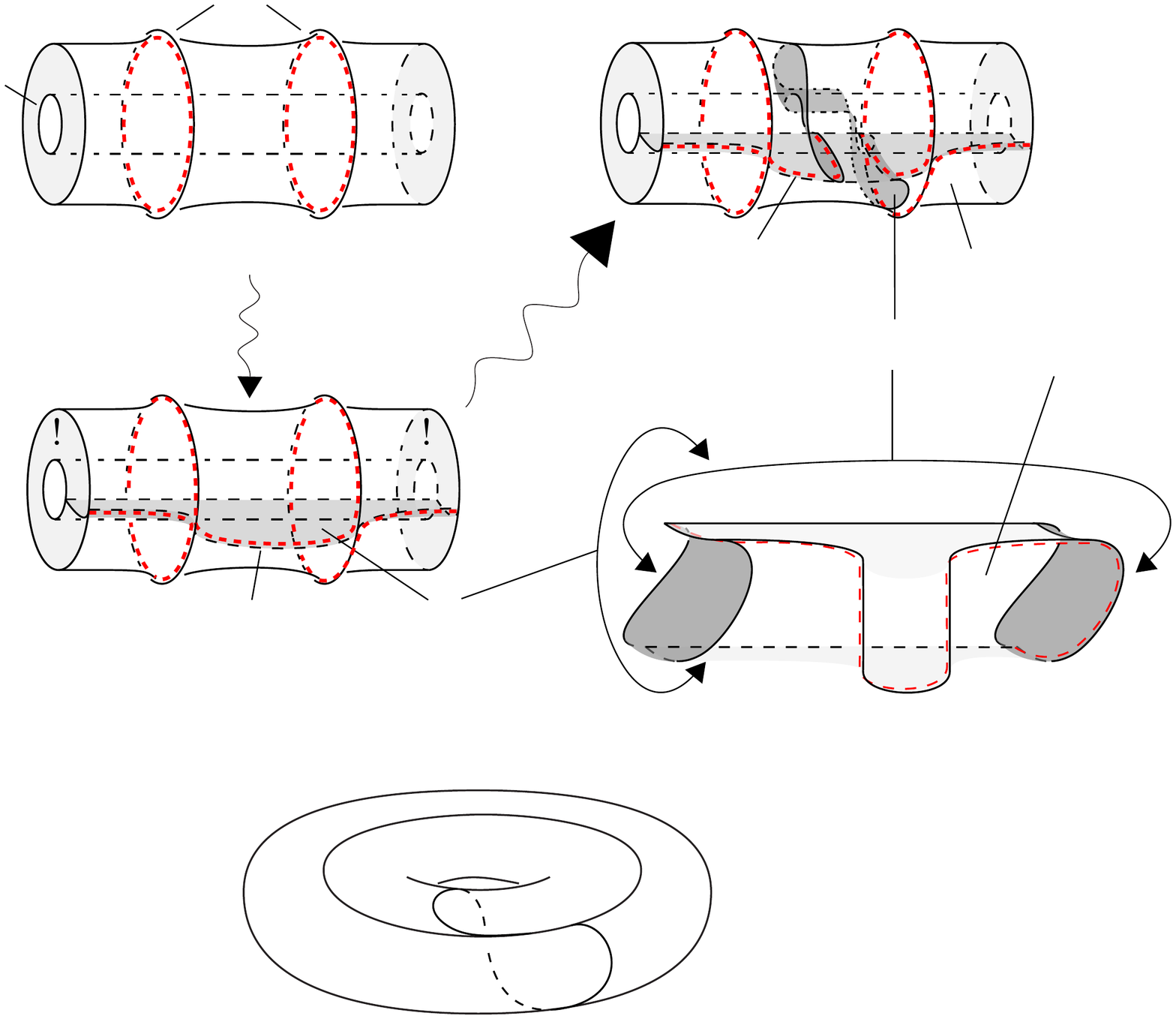}
\end{center}
\caption{Filling $Y_\partial$ to fully realize all boundary slopes.}  \label{figure: realizes all slopes}
\end{figure}

 Filling in the product complementary regions of the resulting lamination with parallel copies of the boundary leaves yields a co-oriented   foliation $\mathcal F_{\gamma}$ that strongly realizes $\gamma$. Since  $\mathcal F_{\gamma}$  has no compact leaves, it is necessarily taut. Since any leaf carried by $B$ has nonempty intersection with every branch of $B$, $\mathcal F_{\gamma}$ has exactly one minimal set. When the surgery coefficient is rational but not an integer, the minimal set of $\mathcal F_{\gamma}$ remains genuine after Dehn filling by slope $\gamma$. 
\end{proof}

This extension of $\mathcal L$ to the family of co-oriented taut foliations $\mathcal F_{\gamma}$ (and $\widehat{\mathcal F}_{\gamma}$) is an extension of the well known ``stacking chair" construction (see, for example, Example~1.1.i in \cite{Gabproblems}). An alternate approach to moving from the lamination $\mathcal L$ to co-oriented taut foliations  in $\widehat{X}_{\gamma}$, for $\gamma$ rational, can be found as Operations~2.3.2 and 2.4.4 in \cite{Gasusp}.

 We note, for the reader interested in understanding all co-oriented taut foliations in the complement of $\kappa$, that there are multiple distinct choices of compatible co-orientations  on $\Sigma$ leading to branched surfaces that fully carry taut foliations. 
 
\section{Additional constructions when the monodromy of $\kappa$ is  neither right- nor left-veering.}\label{Remaining}

Recall that if the monodromy of $\kappa \in S^3$ is neither right- nor left-veering, Theorem~\ref{oldresult} guarantees that any nontrivial slope is strongly realized by some co-oriented taut foliation.  We now introduce several new constructions  of co-oriented taut foliations that give the same result, most of which differ from the  foliations of \cite{Rfib1,Rfib2} in that they  have genuine minimal set. 
 
We note in passing that if the monodromy of $\kappa$ is neither right- nor left-veering, then  it has fractional Dehn twist coefficient zero \cite{HKM}, or, equivalently, Gabai degeneracy $n\cdot \frac{1}{0}$ for some $n\ge 1$ \cite{KR}.

\begin{lemma}
A connected sum of fibered knots in $S^3$ has right-veering (respectively, left-veering) monodromy if and only if each of its components has right-veering (respectively, left-veering) monodromy.
\end{lemma}

\begin{proof}
By induction, it suffices to consider the case of two nontrivial summands. The result follows immediately from Corollary~1.4 of \cite{GabMurasugi2}, or, more directly, from an analysis of product disks.
\end{proof}

\begin{cor}
Suppose $\kappa=\kappa_1\#\kappa_2$ is a fibered knot in $S^3$.  If the monodromy is neither right- nor left-veering, then one of the following must be true:
\begin{enumerate}
\item at least one of $\kappa_1$ or $\kappa_2$ has monodromy that is neither right- nor left-veering, or
\item one of $\kappa_1$ or $\kappa_2$ has monodromy that is  right-veering, and the other summand is  left-veering. \qed
\end{enumerate}
\end{cor}

We proceed as in the right-veering case, by first constructing a spine, and then describing a smoothing by fixing a compatible choice of co-orientations on the branches of this spine.  We note for completeness that we could address each summand separately in the manner of Section~\ref{proofofmaintheorem}:  as in Section~\ref{s: the spine}, let $\Sigma=T \cup A\cup F_1\cup F_2 \cup D_1 \cup D_2$, with the orientations on $A, T, F_1, F_2, \alpha_1$, and $\alpha_2$ chosen as before in Section~\ref{s: firstsmoothing}.  The only difference in the case that  $D_i$ has transition arcs of opposite sign is that, along with a local smoothing of Type C, we see a local smoothing as shown in Figures \ref{negA} and \ref{posA}, which we call \emph{Type A};  again, as before, the sutures of $B$ agree with those of $B^G$, as shown in Figure \ref{suturesstillwork}.

\begin{figure}[ht]
\labellist
\small
\pinlabel {\tiny $\mp$} at 60 540
\pinlabel {\tiny $\pm$} at 105 530
\pinlabel {\tiny $\pm$} at 105 492
\pinlabel {\tiny $\pm$} at 105 463
\pinlabel {\tiny $\pm$} at 149 453
\pinlabel {\tiny $\mp$} at 285 516
\pinlabel {\tiny $\pm$} at 315 530
\pinlabel {\tiny $\pm$} at 315 493
\pinlabel {\tiny $\pm$} at 314 469
\pinlabel {\tiny $\pm$} at 342 474
\pinlabel {\tiny $\mp$} at 280 387
\pinlabel {\tiny $\pm$} at 315 400
\pinlabel {\tiny $\pm$} at 315 359
\pinlabel {\tiny $\pm$} at 315 340
\pinlabel {\tiny $\pm$} at 342 341
\pinlabel {\footnotesize Triple point} at 428 415
\pinlabel {\footnotesize Branch} [Bl] at 165 332
\pinlabel {\footnotesize curves} [Bl] at 165 310
\pinlabel $T$ at 625 420
\pinlabel $F$ at 695 382
\pinlabel $D$ at 600 365
\pinlabel $T$ at 672 347
\pinlabel $F$ at 640 335
\pinlabel $T$ at 570 325
\pinlabel {\tiny $\mp$} at 53 278
\pinlabel {\tiny $\mp$} at 100 193
\pinlabel {\tiny $\mp$} at 100 265
\pinlabel {\tiny $\mp$} at 100 235
\pinlabel {\tiny $\pm$} at 142 185
\pinlabel {\tiny $\mp$} at 277 252
\pinlabel {\tiny $\mp$} at 307 197
\pinlabel {\tiny $\mp$} at 310 260
\pinlabel {\tiny $\mp$} at 310 235
\pinlabel {\tiny $\pm$} at 335 205
\pinlabel {\tiny $\mp$} at 277 119
\pinlabel {\tiny $\mp$} at 307 64
\pinlabel {\tiny $\mp$} at 310 129
\pinlabel {\tiny $\mp$} at 310 102
\pinlabel {\tiny $\pm$} at 338 70
\pinlabel {\footnotesize Branch} [Bl] at 168 50
\pinlabel {\footnotesize curves} [Bl] at 168 28
\endlabellist
\begin{center}
\includegraphics[scale=.4]{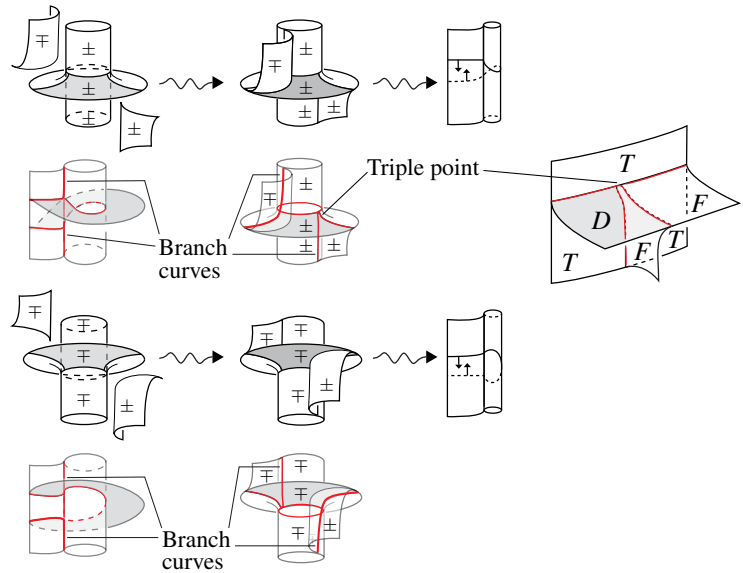}
\end{center}
\caption{$B$ in a neighbourhood of a negative Type A transition.}
\label{negA}
\end{figure}

\begin{figure}[ht]
\labellist
\small
\pinlabel {\tiny $\mp$} at 230 555
\pinlabel {\tiny $\pm$} at 185 545
\pinlabel {\tiny $\pm$} at 185 505
\pinlabel {\tiny $\pm$} at 185 475
\pinlabel {\tiny $\pm$} at 137 463
\pinlabel {\tiny $\mp$} at 422 535
\pinlabel {\tiny $\pm$} at 392 545
\pinlabel {\tiny $\pm$} at 395 505
\pinlabel {\tiny $\pm$} at 395 480
\pinlabel {\tiny $\pm$} at 360 485
\pinlabel {\tiny $\mp$} at 425 389
\pinlabel {\tiny $\pm$} at 392 400
\pinlabel {\tiny $\pm$} at 395 360
\pinlabel {\tiny $\pm$} at 395 335
\pinlabel {\tiny $\pm$} at 360 343
\pinlabel {\footnotesize Branch} [Bl] at 256 421
\pinlabel {\footnotesize curves} [Bl] at 256 399
\pinlabel {\tiny $\mp$} at 235 285
\pinlabel {\tiny $\mp$} at 190 272
\pinlabel {\tiny $\mp$} at 190 243
\pinlabel {\tiny $\mp$} at 190 205
\pinlabel {\tiny $\pm$} at 147 198
\pinlabel {\tiny $\mp$} at 425 260
\pinlabel {\tiny $\mp$} at 392 265
\pinlabel {\tiny $\mp$} at 395 240
\pinlabel {\tiny $\mp$} at 398 210
\pinlabel {\tiny $\pm$} at 370 220
\pinlabel {\tiny $\mp$} at 425 120
\pinlabel {\tiny $\mp$} at 392 122
\pinlabel {\tiny $\mp$} at 395 100
\pinlabel {\tiny $\mp$} at 400 65
\pinlabel {\tiny $\pm$} at 367 73
\pinlabel {\footnotesize Branch} [Bl] at 256 135
\pinlabel {\footnotesize curves} [Bl] at 256 114
\endlabellist
\begin{center}
\includegraphics[scale=.4]{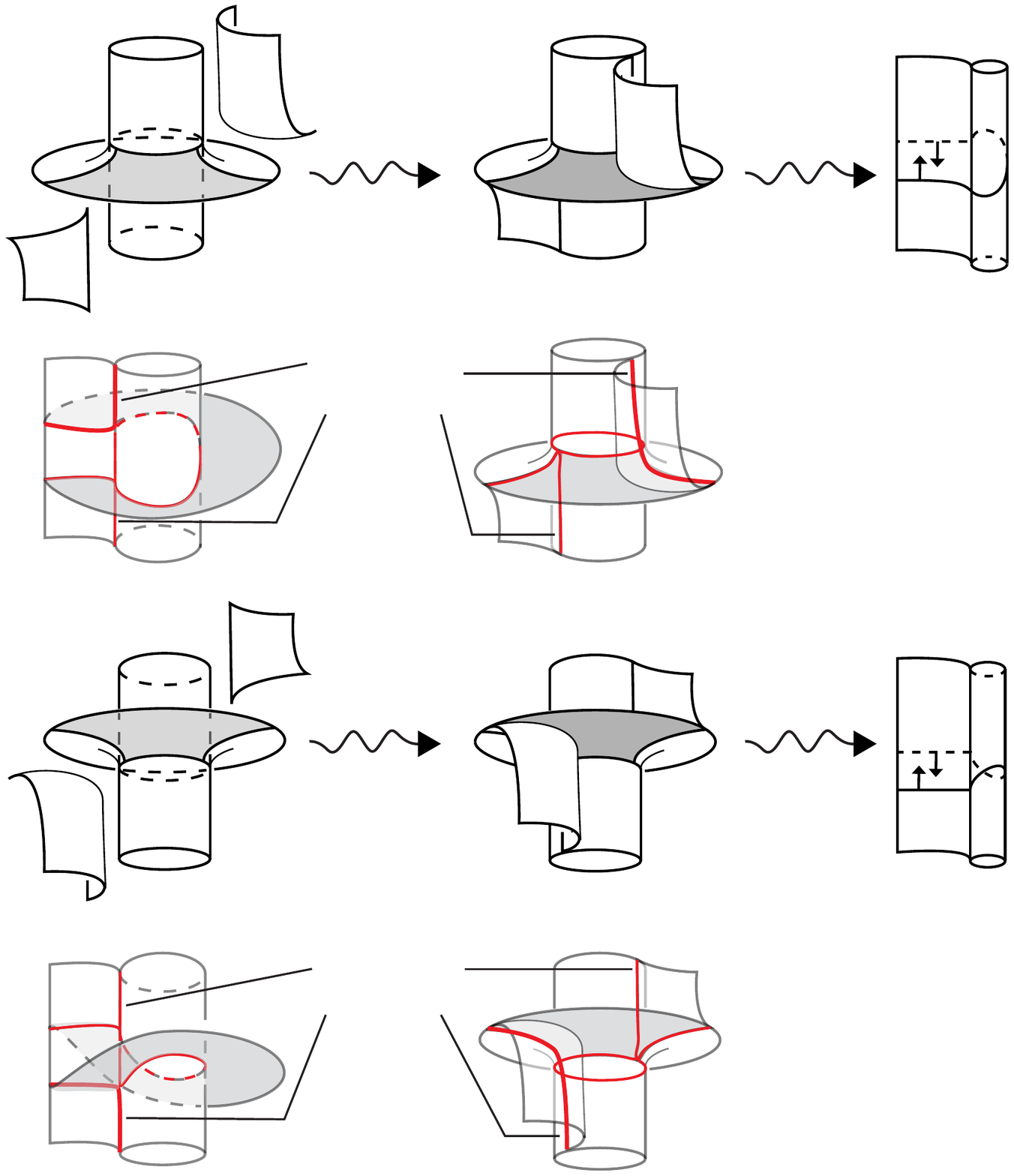}
\end{center}
\caption{$B$ in a neighbourhood of a positive Type A transition.}
\label{posA}
\end{figure}

\begin{figure}[t]
\labellist
\small
\pinlabel {Type C} [B] at 355 410
\pinlabel {Type A} [B] at 462 410
\pinlabel $\alpha$ [Br] at 325 260
\pinlabel $\phi(\alpha)$ [Br] at 387 260
\pinlabel $\phi(\alpha)$ [Br] at 452 260
\pinlabel $\alpha$ [Br] at 489 260
\endlabellist
\begin{center}
\includegraphics[scale= .4]{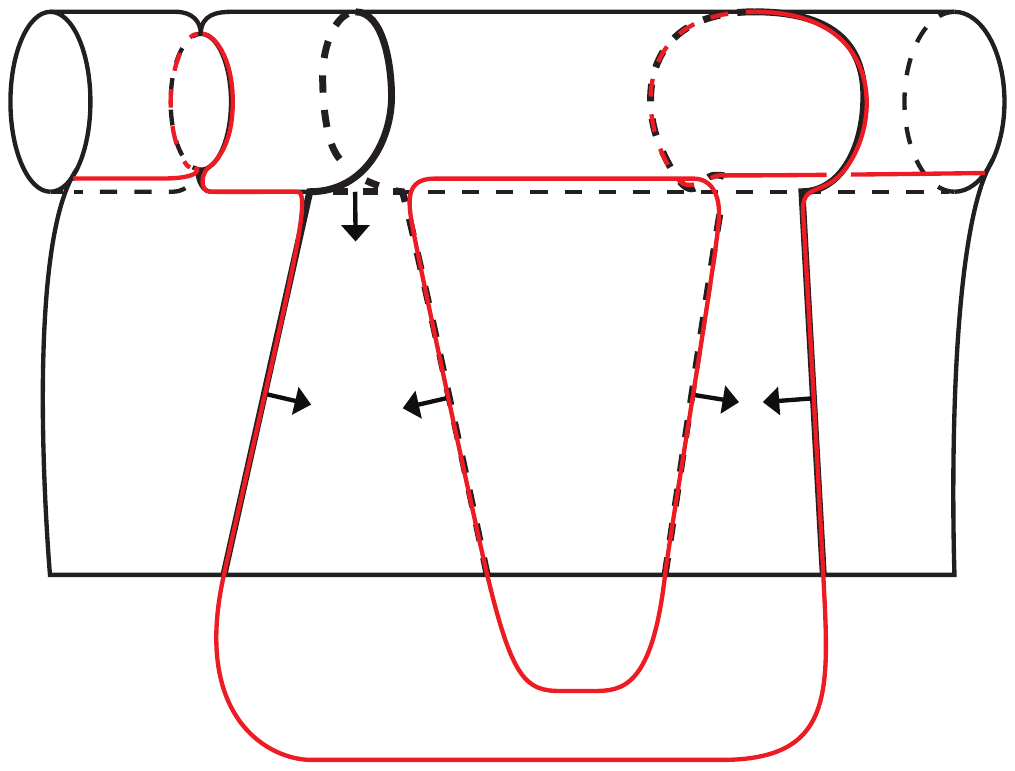}
\end{center}
\caption{The sutures of   $B$  agree with those of $B^G$.}\label{suturesstillwork}
\end{figure}

However,  there is a simpler and more general construction which does not depend on having a connected sum,  given in the following proposition:

\begin{thm} \label{richcase} 
Suppose that $X$ is a fibered 3-manifold, with  fiber $F$ a compact oriented surface with connected boundary, and orientation-preserving monodromy $\phi$.  If there is a tight arc $\alpha$ so that the corresponding product disk $D(\alpha)$ has transition arcs of opposite sign, then there   is a  co-oriented taut foliation $\mathcal F_{\gamma}$ that strongly realizes slope $\gamma$ for all slopes except   $\mu$, the distinguished meridian. The foliation $\mathcal F_{\gamma}$ has a unique  minimal set, and this minimal set is genuine and  disjoint from $\partial M$. Furthermore, each $\mathcal F_{\gamma}$ extends to a co-oriented taut  foliation $\widehat{\mathcal F}_{\gamma}$ in $\widehat{X}(\gamma)$, the closed 3-manifold obtained by Dehn filling along $\gamma$, and  when  $\gamma$ intersects the meridian efficiently in at least two points, the minimal set of $\widehat{\mathcal F}(\gamma)$  is genuine   as well.

\end{thm}

\begin{figure}[ht]
\labellist
\small
\pinlabel {Type C} [B] at 355 410
\pinlabel {Type C} [B] at 462 410
\pinlabel $\alpha$ [Br] at 325 260
\pinlabel $\phi(\alpha)$ [Br] at 387 260
\pinlabel $\phi(\alpha)$ [Br] at 452 260
\pinlabel $\alpha$ [Br] at 489 260
\endlabellist
\begin{center}
\includegraphics[scale=.4]{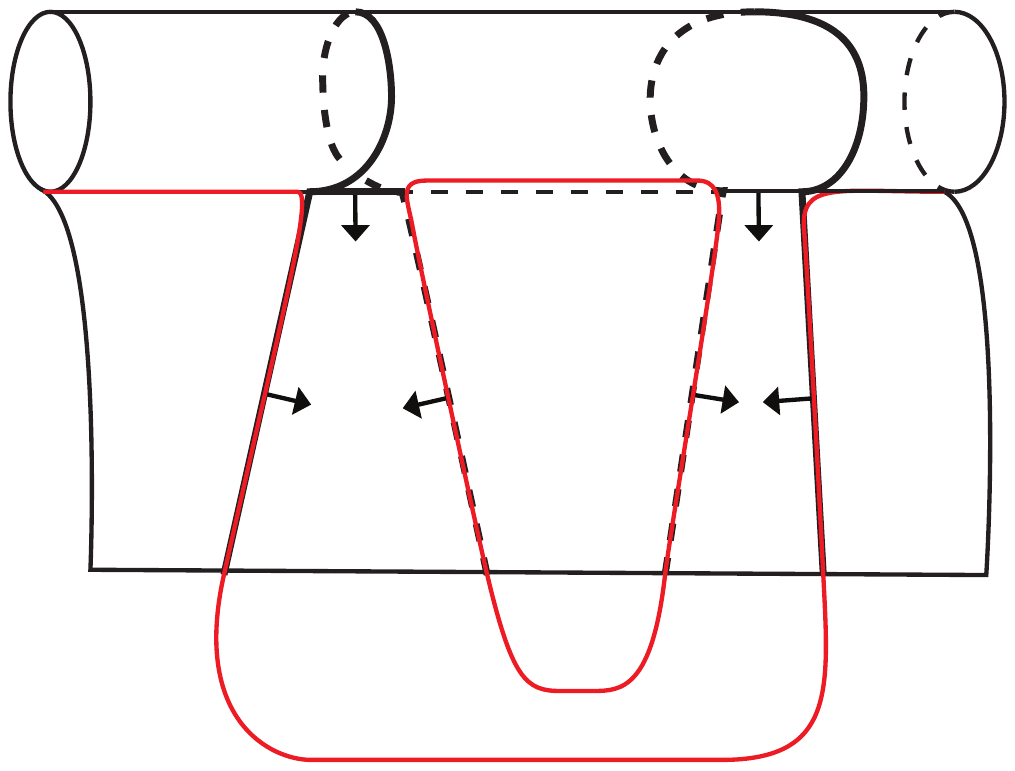}
\end{center}
\caption{Introducing two meridian cusps.}
\label{twotypeC}
\end{figure}

\begin{proof}
Set $D=D(\alpha)$, $\mu_0 = \mu_0(\alpha)$, and $\mu_1 = \mu_1(\alpha)$. Choose a small annular neighbourhood of $\partial F$ in $F$, and let $F_0$ denote the complement of this annulus in $F$. We may assume $\phi$ restricts to a homeomorphism of $F_0$; so $(F_0\times I)/\phi$ is a codimension zero submanifold of $X$. Let $T$ denote the torus boundary of this submanifold.

Now consider the spine  $T\cup F_0\cup D$. Fix an arbitrary co-orientation on $F_0$, and choose the co-orientation on $D$ that results in the smoothing in the interior of $F_0$ that is indicated in Figure~\ref{twotypeC}. 

Now choose co-orientations on the components of  $T|_{\mu_0\cup\mu_1}$ in a neighbourhood of the spine about each transition  so that a meridian cusp in introduced at each, as modelled in Figures \ref{cuspintropos} and \ref{cuspintro}.   Since the transition arcs are of opposite sign, there is a compatible choice of co-orientation on the components of   $T|_{\mu_0\cup\mu_1}$  that agrees with these local choices.  As illustrated in Figure~\ref{twotypeC}, the complementary region of the resulting branched surface that does not contain $\partial X_\kappa$ is isomorphic as a sutured manifold to $B^G(F_0,D)$.  We thus obtain a branched surface $B$ with one  complementary region  homeomorphic to a $(F_0'\times I,\partial F_0'\times I)$, where $F_0'=F_0|_{\alpha}$, and one   complementary region homomorphic to   $\left(\partial X_\kappa \times I, V_0 \cup V_1 \right)$, where $V_0$ and $V_1$ are disjoint meridional annuli with cores $\mu_0$ and $\mu_1$, respectively. It is therefore essential. Apply  the arguments of Section~\ref{proofofmaintheorem} to the splitting of $B$ guaranteed by   Theorem~\ref{Gcanbemadelaminar}  to see that $B$ splits  to a laminar branched surface. The desired conclusions now follow as in our previous constructions.
\end{proof}

\bibliographystyle{amsplain}

\end{document}